\documentclass[a4paper,dvipdfmx]{amsart}
\usepackage{amssymb}
\usepackage{verbatim}
\usepackage[dvipdfmx]{graphicx}
\usepackage{tikz}
\usetikzlibrary{intersections,calc,arrows.meta,patterns}
\theoremstyle{plain}
  \newtheorem{thm}{Theorem}[section]
  \newtheorem{lem}[thm]{Lemma}
  \newtheorem{cor}[thm]{Corollary}
  \newtheorem{prop}[thm]{Proposition}
  \newtheorem{conj}[thm]{Conjecture}

  \newtheorem*{obs*}{Observation}
\theoremstyle{definition}
  \newtheorem{defn}[thm]{Definition}

\theoremstyle{remark}
  \newtheorem{rem}[thm]{Remark}

\newcommand{\Z}{\mathbb{Z}}
\newcommand{\C}{\mathbb{C}}
\newcommand{\R}{\mathbb{R}}

\newcommand{\Vol}{\operatorname{Vol}}

\newcommand{\CS}{\operatorname{CS}}
\newcommand{\Li}{\operatorname{Li}}
\newcommand{\Hom}{\operatorname{Hom}}
\renewcommand{\L}{\mathcal{L}}
\newcommand{\cs}{\operatorname{cs}}
\renewcommand{\sl}{\mathfrak{sl}}
\newcommand{\SL}{\rm{SL}}
\newcommand{\Tr}{\operatorname{Tr}}
\newcommand{\Tor}{\mathbb{T}}
\newcommand{\arccosh}{\operatorname{arccosh}}
\newcommand{\Res}{\operatorname{Res}}

\renewcommand{\Re}{\operatorname{Re}}
\renewcommand{\Im}{\operatorname{Im}}

\renewcommand{\i}{\sqrt{-1}}
\numberwithin{equation}{section}
\allowdisplaybreaks
\begin{document}
\title[The colored Jones polynomial of the figure-eight knot]
{On the asymptotic behavior of the colored Jones polynomial of the figure-eight knot associated with a real number}
\author{Hitoshi Murakami}
\address{
Graduate School of Information Sciences,
Tohoku University,
Aramaki-aza-Aoba 6-3-09, Aoba-ku,
Sendai 980-8579, Japan}
\email{hitoshi@tohoku.ac.jp}
\author{Anh T.~Tran}
\address{
Department of Mathematical Sciences, The University of Texas at Dallas, Richardson,
TX 75080, USA}
\email{att140830@utdallas.edu}
\date{\today}
\dedicatory{Dedicated to the memory of Toshie Takata}
\begin{abstract}
We study the asymptotic behavior of the $N$-dimensional colored Jones polynomial evaluated at $\exp(\xi/N)$ for a real number $\xi$ greater than a certain constant.
We prove that, from the asymptotic behavior, we can extract the $\SL(2;\C)$ Chern--Simons invariant and the Reidemeister torsion twisted by the adjoint action both associated with a representation determined by $\xi$.
\end{abstract}
\thanks{The first author is supported by JSPS KAKENHI Grant Numbers JP16H03927, JP22H01117, JP20K03601, JP20K03931.
The second author is supported by grants from the Simons Foundation (\#354595 and \#708778).}
\maketitle
\section{Introduction}\label{sec:introduction}
Let $N\ge2$ be an integer.
\par
In \cite{Kashaev:MODPLA95}, R.~Kashaev introduced a link invariant $\langle K\rangle_{N}$ for a knot $K$ in the three-sphere $S^3$ by using the so-called quantum dilogarithm.
In \cite{Kashaev:LETMP97}, he studied its asymptotic behavior when $N\to\infty$ for several hyperbolic knots and conjectured that it would determine its hyperbolic volume for any {\em hyperbolic} knot, where a knot is called hyperbolic if its complement possesses a complete hyperbolic structure with finite volume.
More precisely he conjectured
\begin{conj}[Kashaev's conjecture]
Let $H$ be a {\em hyperbolic} knot in $S^3$.
Then the following equality holds.
\begin{equation}\label{eq:Kashaev}
  \lim_{N\to\infty}
  \frac{\log\left|\langle H\rangle_{N}\right|}{N}
  =
  \frac{\operatorname{V}\left(S^3\setminus{H}\right)}{2\pi},
\end{equation}
where $\operatorname{V}\left(S^3\setminus{H}\right)$ is the hyperbolic volume.
\end{conj}
\par
Let $J_{N}(K;q)$ be the colored Jones polynomial associated with the $N$-dimensional irreducible representation of the Lie algebra $\sl(2;\C)$ \cite{Kirby/Melvin:INVEM1991,Reshetikhin/Turaev:INVEM1991}.
We normalize it so that $J_{N}(U;q)=1$ for the unknot $U$, and that $J_{2}(K;q)$ is the ordinary Jones polynomial \cite{Jones:BULAM31985}.
In \cite{Murakami/Murakami:ACTAM12001}, J.~Murakami and the first author proved that Kashaev's invariant $\langle K\rangle_{N}$ coincides with the $N$-dimensional colored Jones polynomial $J_{N}(K;q)$  evaluated at $q=e^{2\pi\i/N}$.
They also generalized Kashaev's conjecture for general knots:
\begin{conj}[Volume Conjecture]
For {\em any} knot $K\subset{S^3}$, we have
\begin{equation}\label{eq:VC}
  \lim_{N\to\infty}
  \frac{\log\left|J_{N}\left(K;e^{2\pi\i/N}\right)\right|}{N}
  =
  \frac{\Vol\left(S^3\setminus{K}\right)}{2\pi}.
\end{equation}
Here, $\Vol(S^3\setminus{K})$ is the simplicial volume, also known as the Gromov norm \cite{Gromov:INSHE82} {\rm(}see also \cite{Soma:INVEM1981}{\rm)}, which is normalized so that $\Vol\left(S^3\setminus{H}\right)=\operatorname{V}\left(S^3\setminus{H}\right)$ if $H$ is hyperbolic.
\end{conj}
\par
Kashaev's conjecture was complexified by T.~Takata, J.~Murakami, M.~Okamoto, Y.~Yokota and the first author in \cite{Murakami/Murakami/Okamoto/Takata/Yokota:EXPMA02} by dropping the absolute value symbol (and replacing $\langle H\rangle_{N}$ with $J_{N}\left(H;e^{2\pi\i/N}\right)$) in \eqref{eq:Kashaev}.
\begin{conj}[Complexification of Kashaev's conjecture]
For a hyperbolic knot $H$, we have
\begin{equation}\label{eq:CVC}
  \lim_{N\to\infty}
  \frac{\log J_{N}\left(H;e^{2\pi\i/N}\right)}{N}
  =
  \frac{\operatorname{V}\left(S^3\setminus{H}\right)+\i\CS(S^3\setminus{H})}{2\pi},
\end{equation}
where $\CS$ is the $SO(3)$ Chern--Simons invariant of $H$ associated with the Levi-Civita connection \cite{Meyerhoff:LMSLN112}.
\end{conj}
\par
More detailed asymptotic formulas are known for several knots \cite{Andersen/Hansen:JKNOT2006,Ohtsuki:QT2016,Ohtsuki/Yokota:MATPC2018}:
\begin{equation*}
  \langle H\rangle_{N}
  \underset{N\to\infty}{\sim}
  \omega(H)N^{3/2}\exp\left(\frac{N}{2\pi}\operatorname{CV}(H)\right),
\end{equation*}
where $\operatorname{CV}(H):=\Vol\left(S^3\setminus{H}\right)+\i\CS(S^3\setminus{H})$ is the complex volume and $2\i\omega(H)^2$ is the adjoint (cohomological) Reidemeister torsion twisted by the holonomy representation of $\pi_1(S^3\setminus{H})$ to $\SL(2;\C)$ \cite{Ohtsuki/Takata:GEOTO2015}.
\par
It was also generalized by replacing $2\pi\i$ with another complex number $\xi$ \cite{Dimofte/Gukov/Lenells/Zagier:CNTP2010,Murakami:JTOP2013}.
\begin{conj}[Generalized Volume Conjecture]
Let $\xi\ne2\pi\i$ be a complex number close to $2\pi\i$, and $H$ a hyperbolic knot.
Then we have
\begin{equation}\label{eq:GVC}
  J_{N}\left(H;e^{\xi/N}\right)
  \underset{N\to\infty}{\sim}
  \frac{C}{\sinh(\xi)}\times\tau(\xi)^{1/2}\left(\frac{N}{\xi}\right)^d
  \exp\left(\frac{N}{\xi}S(\xi)\right)
\end{equation}
with $d$ a constant.
Here $S(\xi)$ and $\tau(\xi)$ are related to the $\SL(2;\C)$ Chern--Simons invariant and the adjoint Reidemeister torsion, respectively, both associated with a certain representation of $\pi_1(S^3\setminus{H})$ to $\SL(2;\C)$.
\end{conj}
\par
In this paper, we are mainly interested in the figure-eight knot.
We first list known results for the asymptotic behavior of the colored Jones polynomial of the figure-eight knot $E$.
\par
We define a function $\varphi$ as
\begin{equation*}
  \varphi(\xi)
  :=
  \arccosh\left(\cosh(\xi)-\frac{1}{2}\right)
\end{equation*}
for a complex number $\xi$, where we use the following branch of $\arccosh$:
\begin{equation*}
  \arccosh(x)
  :=
  \log\left(x-\i\sqrt{1-x^2}\right),
\end{equation*}
and we choose the branch cut of $\log$ as $(-\infty,0)$.
Note that $\varphi(0)=-\pi\i/3$ and $\varphi(\kappa)=0$, where $\kappa:=\log\left(\frac{3+\sqrt{5}}{2}\right)$.
We also define
\begin{align*}
  S(\xi)
  &:=
  \Li_2(e^{-\xi-\varphi(\xi)})-\Li_2(e^{-\xi+\varphi(\xi)})+\xi\varphi(\xi),
  \\
  \tilde{S}(\xi)
  &:=
  \Li_2(e^{-\xi-\varphi(\xi)})-\Li_2(e^{-\xi+\varphi(\xi)})
  +(\xi-2\pi\i)(\varphi(\xi)+2\pi\i),
  \\
  T(\xi)
  &:=
  \frac{2}{\sqrt{(2\cosh(\xi)+1)(2\cosh(\xi)-3)}},
\end{align*}
where
\begin{equation*}
  \Li_2(z):=-\int_{0}^{z}\frac{\log(1-x)}{x}\,dx
\end{equation*}
is the dilogarithm function, where we choose the branch cut as $(1,\infty)$.
Note that $S(0)=\tilde{S}(2\pi\i)=\Li_2(e^{\pi\i/3})-\Li_2(e^{-\pi\i/3})=\i\times 2.02988...$.
\par
The quantities $S(\xi)$, $\tilde{S}(\xi)$, and $T(\xi)$ are related to the Chern--Simons invariant and the adjoint Reidemeister torsion associated with a certain representation of $\pi_1\left(S^3\setminus{E}\right)$ to $\SL(2;\C)$.
See Section~\ref{sec:CST} for details.
\begin{enumerate}
\item
$\xi=2\pi\i$:
\par
This corresponds to the case of the Volume Conjecture.
Kashaev sketched a proof of \eqref{eq:VC} in \cite{Kashaev:LETMP97}, and then T.~Ekholm gave a detailed proof of \eqref{eq:VC}.
See for example \cite[Section~3.2]{Murakami/Yokota:2018} for Ekholm's proof.
Later, J.E.~Andersen and S.K.~Hansen \cite[Theorem~1]{Andersen/Hansen:JKNOT2006} followed Kashaev's method to prove the following asymptotic formula.
\begin{equation*}
  J_{N}\left(E;e^{2\pi\i/N}\right)
  \underset{N\to\infty}{\sim}
  2\pi^{3/2}\left(\frac{N}{2\pi\i}\right)^{3/2}T(2\pi\i)^{1/2}
  \exp\left(\frac{N}{2\pi\i}\tilde{S}(2\pi\i)\right).
\end{equation*}
Since $S(0)=\i\Vol\left(S^3\setminus{E}\right)$, this refines the Volume Conjecture for the figure-eight knot.
\item
$\xi$ is close to $2\pi\i$ and not purely imaginary:
\par
Y.~Yokota and the first author proved the following formula \cite{Murakami/Yokota:JREIA2007}.
\begin{equation}\label{eq:Yokota}
  \lim_{N\to\infty}
  \frac{\log J_{N}(E;e^{\xi/N})}{N}
  =
  \frac{\tilde{S}(\xi)}{\xi}.
\end{equation}
\item
$\xi$ is purely imaginary with $5\pi/3<|\xi|<7\pi/3$, and $2\pi/|\xi|$ is irrational with finite irrationality measure:
\par
In \cite{Murakami:KYUMJ2004} (see \cite{Murakami:KYUMJ2016} for correction), the first author proved the following formula.
\begin{equation*}
  \lim_{N\to\infty}
  \frac{\log J_{N}(E;e^{\xi/N})}{N}
  =
  \frac{\tilde{S}(\xi)}{\xi}.
\end{equation*}
See also \cite[6.2.2]{Murakami:Novosibirsk}.
\item
$\xi$ is of the form $2\pi\i+u$ for a real number $u$ with $0<|u|<\kappa$:
\par
The first author proved the following asymptotic formula \cite[Theorem~1.4]{Murakami:JTOP2013}:
\begin{equation}\label{eq:J_top}
  J_{N}\left(E;e^{\xi/N}\right)
  \underset{N\to\infty}{\sim}
  \frac{\sqrt{-\pi}}{2\sinh(u/2)}T(\xi)^{1/2}\left(\frac{N}{\xi}\right)^{1/2}
  \exp\left(\frac{N}{\xi}\tilde{S}(\xi)\right).
\end{equation}
Note that this refines \eqref{eq:Yokota} when $\xi$ is as above.
\item
$\xi$ satisfies the inequalities $|2\cosh{\xi}-2|<1$ and $|\Im{\xi}|<\pi/3$:
\par
In this case, $J_{N}\left(E;e^{\xi/N}\right)$ converges.
In fact, we have
\begin{equation*}
  \lim_{N\to\infty}
  J_{N}\left(E;e^{\xi/N}\right)
  =
  \frac{1}{\Delta(E;e^{\xi})},
\end{equation*}
where $\Delta(E;t)=-t+3-t^{-1}$ is the normalized Alexander polynomial of the figure-eight knot $E$
\cite[Theorem~1.1]{Murakami:JPJGT2007}.
\par
See \cite{Garoufalidis/Le:GEOTO2011} for general knots.
\item
$\xi=\pm\kappa$:
\par
In this case, $J_{N}\left(E;e^{\kappa/N}\right)$ grows polynomially.
More precisely, we have
\begin{equation*}
  J_{N}\left(E;e^{\kappa/N}\right)
  \underset{N\to\infty}{\sim}
  \frac{\Gamma(1/3)}{(3\kappa)^{2/3}}N^{2/3},
\end{equation*}
where $\Gamma(z)$ is the Gamma function \cite[Theorem~1.1]{Hikami/Murakami:COMCM2008}.
\item
$\xi$ is real and $|\xi|>\kappa$:
\par
The first author proved the following formula \cite[Theorem~8.1]{Murakami:KYUMJ2004} (see also \cite[Theorem~3.2]{Murakami:ACTMV2008} and \cite[Lemma~6.7]{Murakami:Novosibirsk}).
\begin{equation}\label{eq:Kyungpook}
  \lim_{N\to\infty}
  \frac{\log J_{N}(E;e^{\xi/N})}{N}
  =
  \frac{S(\xi)}{|\xi|}.
\end{equation}
\end{enumerate}
\par
The purpose of this paper is to refine \eqref{eq:Kyungpook}.
We will show
\begin{thm}\label{thm:main}
If $\xi$ is real and $|\xi|>\kappa$, then we have
\begin{equation*}
  J_{N}\left(E;e^{\xi/N}\right)
  \underset{N\to\infty}{\sim}
  \frac{\sqrt{\pi}}{2\sinh(|\xi|/2)}T(\xi)^{1/2}\left(\frac{N}{|\xi|}\right)^{1/2}
  \exp\left(\frac{N}{|\xi|}S(\xi)\right).
\end{equation*}
\end{thm}
Moreover, we can show that $T(\xi)$ is the Reidemeister torsion twisted by the adjoint action of a representation $\rho$ of $\pi_1(S^3\setminus{E})$ to $\SL(2;\C)$ determined by $\xi$, and that $S(\xi)-\xi\eta/2$ is the Chern--Simons invariant of $\rho$ associated with the meridian and the preferred longitude of $E\subset S^{3}$.
See Section~\ref{sec:CST} for details.
\begin{rem}
Suppose that $u:=\xi-2\pi\i$ is a small complex number or a real number with $|u|<\kappa$.
In this case $\Im\varphi(u)<0$.
By using a well-known formula
\begin{equation*}
  \Li_2(z)+\Li_2(z^{-1})+\frac{\pi^2}{6}+\frac{1}{2}\left(\log(-z)\right)^2=0.
\end{equation*}
we have
\begin{equation*}
\begin{split}
  &\tilde{S}(\xi)
  \\
  =&
  -\Li_2(e^{\xi+\varphi(\xi)})+\Li_2(e^{\xi-\varphi(\xi)})
  -\frac{1}{2}\left(\log(-e^{\xi+\varphi(\xi)})\right)^2
  +\frac{1}{2}\left(\log(-e^{\xi-\varphi(\xi)})\right)^2
  \\
  &+u(\varphi(\xi)+2\pi\i)
  \\
  \\
  =&
  -\Li_2(e^{u+\varphi(u)})+\Li_2(e^{u-\varphi(u)})
  -\frac{1}{2}\left(u+\varphi(u)+\pi\i\right)^2
  \\
  &+\frac{1}{2}\left(u-\varphi(u)-\pi\i\right)^2
   +(2\pi\i+u)\varphi(u)
  \\
  =&
  -\Li_2(e^{u+\varphi(u)})+\Li_2(e^{u-\varphi(u)})-u\varphi(u).
\end{split}
\end{equation*}
In the second equality, we use the fact that $\Im\varphi(u)<0$.
Therefore \eqref{eq:J_top} coincides with the formula appearing in \cite[Theorem~1.4]{Murakami:JTOP2013}.
\end{rem}
\begin{rem}
In \cite{Murakami:JTOP2013}, the first author followed \cite{Andersen/Hansen:JKNOT2006} to obtain the asymptotic formula, but in the current paper we follow \cite{Ohtsuki:QT2016}.
\end{rem}
\section{Preliminaries}
In this section, we first introduce the colored Jones polynomial, and then we define a quantum dilogarithm, and variants of the logarithm and the dilogarithm.
We also describe some of their properties.
\par
For a knot $K$ in the three-sphere $S^3$, we denote by $J_N(K;q)$ the colored Jones polynomial of $K$ associated with the $N$-dimensional irreducible representation of the Lie algebra $\frak{sl}(2;\C)$ \cite{Kirby/Melvin:INVEM1991,Reshetikhin/Turaev:INVEM1991}.
We normalize it so that $J_N(U;q)=1$ for the unknot $U\subset S^3$.
Note that $J_2(K;q)$ is the original Jones polynomial \cite{Jones:BULAM31985}.
\par
Let $E$ be the figure-eight knot.
We use the following formula obtained by K.~Habiro \cite[P.~36 (1)]{Habiro:SURIK2000} and T.T.Q.~Le \cite[P.~129]{Le:TOPOA2003}.
See also \cite{Masbaum:ALGGT12003}.
\begin{equation}\label{eq:J_N_fig8}
\begin{split}
  J_N(E;q)
  &=
  \sum_{k=0}^{N-1}
  \prod_{l=1}^{k}\left(q^{(N-l)/2}-q^{-(N-l)/2}\right)\left(q^{(N+l)/2}-q^{(N+l)/2}\right)
  \\
  &=
  \sum_{k=0}^{N-1}
  q^{-kN}
  \prod_{l=1}^{k}\left(1-q^{N-l}\right)\left(1-q^{N+l}\right).
\end{split}
\end{equation}
\par
Next, we define functions $T_N(z)$, $\L_{0}(z)$, $\L_{1}(z)$, and $\L_{2}(z)$, which are related as
\begin{align*}
  T_{N}(z)
  &=
  \frac{N}{\xi}\L_{2}(z)+O(1/N)\quad(N\to\infty),
  \\
  \frac{d\,\L_{2}}{d\,z}(z)
  &=
  -2\pi\i\L_{1}(z),
  \\
  \frac{d\,\L_{1}}{d\,z}(z)
  &=
  -\L_{0}(z).
\end{align*}
\par
Let $\xi$ be a positive real number.
We define $C_{0}:=(-\infty,-1]\cup\{e^{\tau\i}\mid0\le\tau\le\pi\}\cup[1,\infty)$ and $C_{\theta}:=e^{\i\theta}C_{0}$, where $\theta$ is a positive real number with $\tan\theta<\frac{\pi}{\xi}$ and we orient $C_{0}$ from left to right.
\par
Consider the following integrals:
\begin{equation*}
  \int_{C_{\theta}}\frac{e^{(2z-1)x}}{x\sinh(x)\sinh(\gamma x)}\,dx,
  \quad
  \int_{C_{\theta}}\frac{e^{(2z-1)x}}{x^{m}\sinh(x)}\,dx,
\end{equation*}
where $m=0,1,2$, $\gamma:=\frac{\xi}{2N\pi\i}$ for an integer $N\ge2$, and $z$ is a complex number with $0<\Re(ze^{\i\theta})<\cos\theta$.
Here we follow \cite{Faddeev:LETMP1995} to introduce $T_N(z)$, which plays an important role in the paper.
Note that the set of the poles of the integrands are
\begin{equation*}
  \{k\pi\i\mid k\in\Z\}\cup\{2lN\pi^2/\xi\mid l\in\Z\}
  \quad\text{and}\quad
  \{k\pi\i\mid k\in\Z\},
\end{equation*}
respectively.
Therefore, if $N$ is large enough, then the only pole inside the unit circle centered at the origin is $0\in\C$ and so the path of integral $C_{\theta}$ avoids the poles.
Proofs of their convergences (Lemmas~\ref{lem:converge_T} and \ref{lem:integrals} below) are given in Section~\ref{sec:proofs_lemmas}.
\begin{lem}\label{lem:converge_T}
The integral $\int_{C_{\theta}}\frac{e^{(2z-1)x}}{x\sinh(x)\sinh(\gamma x)}\,dx$ converges if $z\in\C$ satisfies $-\frac{\xi\sin\theta}{4N\pi}<\Re(ze^{\theta\i})<\cos\theta+\frac{\xi\sin\theta}{4N\pi}$.
\end{lem}
\begin{lem}\label{lem:integrals}
If $z\in\C$ satisfies $0<\Re(ze^{\theta\i})<\cos\theta$, then the integral $\int_{C_{\theta}}\frac{e^{(2z-1)x}}{x^{m}\sinh(x)}\,dx$ converges for $m=0,1,2$.
\end{lem}
\par
Now define functions $T_N(z)$ and $\L_{m}(z)$ ($m=0,1,2$) by using the integrals above.
\begin{defn}
Fix an integer $N\ge2$ and put $\gamma:=\frac{\xi}{2N\pi\i}$.
We define
\begin{equation*}
  T_N(z)
  :=
  \frac{1}{4}
  \int_{C_{\theta}}
  \frac{e^{(2z-1)x}}{x\sinh(x)\sinh(\gamma x)}\,dx,
\end{equation*}
for a complex number $z$ with $-\frac{\xi\sin\theta}{4N\pi}<\Re(ze^{\theta\i})<\cos\theta+\frac{\xi\sin\theta}{4N\pi}$.
\end{defn}
\par
\begin{defn}
For a complex number $z$ with $0<\Re(ze^{\i\theta})<\cos\theta$, we define
\begin{align*}
  \mathcal{L}_0(z)
  &:=
  \int_{C_{\theta}}\frac{e^{(2z-1)x}}{\sinh(x)}\,dx,
  \\
  \L_1(z)
  &:=
  -\frac{1}{2}
  \int_{C_{\theta}}\frac{e^{(2z-1)x}}{x\sinh(x)}\,dx,
  \\[5mm]
  \L_2(z)
  &:=
  \frac{\pi\i}{2}
  \int_{C_{\theta}}\frac{e^{(2z-1)x}}{x^2\sinh(x)}\,dx.
\end{align*}
\end{defn}
\begin{lem}\label{lem:L0_1_2}
For $m=0,1,2$, we calculate $\L_{m}(z)$ as follows.
\begin{align}
  \L_0(z)
  &=
  \frac{-2\pi\i}{1-e^{-2\pi\i z}},
  \label{eq:L0}
  \\
  \L_1(z)
  &=
  \begin{cases}
    \log\left(1-e^{2\pi\i z}\right)&\text{if $\Im{z}\ge0$,}
    \\[3mm]
    \pi\i(2z-1)+\log\left(1-e^{-2\pi\i z}\right)&\text{if $\Im{z}<0$,}
  \end{cases}
  \label{eq:L1}
  \\[5mm]
  \L_2(z)
  &=
  \begin{cases}
    \Li_2\left(e^{2\pi\i z}\right)&\text{if $\Im{z}\ge0$,}
    \\[3mm]
    \pi^2\left(2z^2-2z+\frac{1}{3}\right)-\Li_2\left(e^{-2\pi\i z}\right)
    &\text{if $\Im{z}<0$}.
  \end{cases}
  \label{eq:L2}
\end{align}
\end{lem}
Proofs are also given in Section~\ref{sec:proofs_lemmas}.
\begin{rem}
In \eqref{eq:L1} and \eqref{eq:L2}, we use $\log(1-x)$ and $\Li_2(x)$ only for $|x|\le1$ ($x\ne1$).
\end{rem}
Since
\begin{equation*}
\begin{split}
  T_{N}(z-\gamma/2)-T_{N}(z+\gamma/2)
  &=
  \int_{C_{\theta}}
  \frac{e^{(2z-\gamma-1)x}-e^{(2z+\gamma-1)x}}{4x\sinh(x)\sinh(\gamma x)}\,dx
  \\
  &=
  -
  \int_{C_{\theta}}
  \frac{e^{(2z-1)x}}{2x\sinh(x)}\,dx
  \\
  &=
  \L_1(z),
\end{split}
\end{equation*}
we have the following corollary.
\begin{cor}
If $0<\Re(ze^{\theta\i})<\cos\theta$, then we have
\begin{equation}\label{eq:lem}
  \frac{\exp\left(T_{N}(z-\gamma/2)\right)}{\exp\left(T_{N}(z+\gamma/2)\right)}
  =
  1-e^{2\pi\i z}.
\end{equation}
\end{cor}
\par
We will show relations among the functions $T_{N}(z)$ and $\L_{m}(z)$ ($m=0,1,2$).
\par
We can prove that $\frac{1}{N}T_{N}(z)$ uniformly converges to $\frac{1}{\xi}\L_{2}(z)$.
More precisely, we have the following proposition.
See \cite[Proposition~A.1]{Ohtsuki:QT2016}.
\begin{prop}\label{prop:T_N_L_2}
For any positive real number $M$ and a sufficiently small positive real number $\nu$, we have
\begin{equation*}
  T_N(z)
  =
  \frac{N}{\xi}\L_2(z)
  +
  O(1/N)
  \quad(N\to\infty)
\end{equation*}
in the region
\begin{equation*}
  \{z\in\C\mid\nu\le\Re(ze^{\theta\i})\le\cos\theta-\nu,|\Im{z}|\le M\}.
\end{equation*}
In particular, the function $\frac{1}{N}T_N(z)$ uniformly converges to $\frac{1}{\xi}\L_2(z)$ in the region above.
\end{prop}
A proof is given in Section~\ref{sec:proofs_lemmas}.
Since we may regard $\L_2(z)$ as a variant of the dilogarithm from Lemma~\ref{lem:L0_1_2}, the function $T_{N}(z)$ is another quantum dilogarithm.
\par
The functions $\L_{0}(z)$, $\L_{1}(z)$, and $\L_{2}(z)$ are related as follows.
\begin{lem}\label{lem:der_L}
The derivatives of $\L_1(z)$ and $\L_2(z)$ are given as follows:
\begin{align*}
  \frac{d\,\L_2}{d\,z}(z)
  &=
  -2\pi\i\L_1(z),
  \\
  \dfrac{d\,\L_1}{d\,z}(z)
  &=
  -\L_0(z).
\end{align*}
\end{lem}
\begin{proof}
We have
\begin{equation*}
\begin{split}
  \frac{d\,\L_2}{d\,z}(z)
  &=
  \frac{\pi\i}{2}\int_{C_{\theta}}\left(\frac{d}{d\,z}\frac{e^{(2z-1)x}}{x^2\sinh(x)}\right)\,dx
  \\
  &=
  \pi\i\int_{C_{\theta}}\frac{e^{(2z-1)x}}{x\sinh(x)}\,dx
  \\
  &=
  -2\pi\i\L_1(z),
\end{split}
\end{equation*}
and
\begin{equation*}
\begin{split}
  \frac{d\,\L_1}{d\,z}(z)
  &=
  -\frac{1}{2}\int_{C_{\theta}}\left(\frac{d}{d\,z}\frac{e^{(2z-1)x}}{x\sinh(x)}\right)\,dx
  \\
  &=
  -\int_{C_{\theta}}\frac{e^{(2z-1)x}}{\sinh(x)}\,dx
  \\
  &=
  -\L_0(z),
\end{split}
\end{equation*}
completing the proof.
\end{proof}
\section{Summation}
In this section, we use the quantum dilogarithm $T_N(z)$ to express $J_N(E;e^{\xi/N})$ without the product of a sequence.
\par
Since the figure-eight knot $E$ is amphicheiral, its colored Jones polynomial is symmetric, that is, $J_{N}(E;q)=J_{N}(E;q^{-1})$.
So we have $J_{N}(E;e^{\xi/N})=J_{N}(E;e^{-\xi/N})$ and we do not need to consider the case where $\xi<0$.
\par
Recall that we choose a positive real number $\theta$ so that $\tan\theta<\frac{\pi}{\xi}$ and that we put $\gamma:=\frac{\xi}{2N\pi\i}$.
Putting $z:=\frac{\xi}{2\pi\i}(1-l/N)$ ($0<l<N$) in \eqref{eq:lem}, we have
\begin{equation*}
  \frac{\exp\left(T_{N}\left(\frac{\xi}{2\pi\i}(1-l/N)-\frac{\xi}{4N\pi\i}\right)\right)}
       {\exp\left(T_{N}\left(\frac{\xi}{2\pi\i}(1-l/N)+\frac{\xi}{4N\pi\i}\right)\right)}
  =
  1-e^{\xi(1-l/N)}
\end{equation*}
since
\begin{equation*}
  \Re\left(\frac{\xi}{2\pi\i}\left(1-\frac{l}{N}\right)e^{\theta\i}\right)
  =
  \frac{\xi}{2\pi}\left(1-\frac{l}{N}\right)\sin\theta
  <
  \frac{1}{2}\left(1-\frac{l}{N}\right)\cos\theta,
\end{equation*}
which is between $0$ and $\cos\theta$.
Therefore we have
\begin{equation*}
\begin{split}
  \prod_{l=1}^{k}\bigl(1-e^{(N-l)\xi/N}\bigr)
  &=
  \prod_{l=1}^{k}
  \frac{\exp\left(T_{N}\left(\frac{\xi}{2\pi\i}\left(1-\frac{2l+1}{2N}\right)\right)\right)}
       {\exp\left(T_{N}\left(\frac{\xi}{2\pi\i}\left(1-\frac{2l-1}{2N}\right)\right)\right)}
  \\
  &=
  \frac{\exp\left(T_{N}\left(\frac{\xi}{2\pi\i}(1-\frac{2k+1}{2N})\right)\right)}
       {\exp\left(T_{N}\left(\frac{\xi}{2\pi\i}(1-\frac{1}{2N})\right)\right)}.
\end{split}
\end{equation*}
\par
Similarly, putting $z:=\frac{\xi}{2\pi\i}(1+l/N)$ ($0<l<N$) in \eqref{eq:lem}, we have
\begin{equation*}
  \frac{\exp
        \left(
          T_{N}
          \left(
            \frac{\xi}{2\pi\i}(1+l/N)-\frac{\xi}{4N\pi\i}
          \right)
        \right)}
       {\exp
        \left(
          T_{N}
          \left(
            \frac{\xi}{2\pi\i}(1+l/N)+\frac{\xi}{4N\pi\i}
          \right)
        \right)}
  =
  1-e^{\xi(1+l/N)}
\end{equation*}
since
\begin{equation*}
  \Re\left(\frac{\xi}{2\pi\i}\left(1+\frac{l}{N}\right)e^{\theta\i}\right)
  =
  \frac{\xi}{2\pi}\left(1+\frac{l}{N}\right)\sin\theta
  <
  \frac{1}{2}\left(1+\frac{l}{N}\right)\cos\theta,
\end{equation*}
which is between $0$ and $\cos\theta$.
So we have
\begin{equation*}
\begin{split}
  \prod_{l=1}^{k}\bigl(1-e^{(N+l)\xi/N}\bigr)
  &=
  \prod_{l=1}^{k}
  \frac{\exp\left(T_{N}\left(\frac{\xi}{2\pi\i}\left(1+\frac{2l-1}{2N}\right)\right)\right)}
       {\exp\left(T_{N}\left(\frac{\xi}{2\pi\i}\left(1+\frac{2l+1}{2N}\right)\right)\right)}
  \\
  &=
  \frac{\exp\left(T_{N}\left(\frac{\xi}{2\pi\i}(1+\frac{1}{2N})\right)\right)}
       {\exp\left(T_{N}\left(\frac{\xi}{2\pi\i}(1+\frac{2k+1}{2N})\right)\right)}.
\end{split}
\end{equation*}
Therefore we have from \eqref{eq:J_N_fig8}
\begin{equation*}
\begin{split}
  &J_N\bigl(E;\exp(\xi/N)\bigr)
  \\
  &=
  \frac{\exp\left(T_{N}\left(\frac{\xi}{2\pi\i}(1+\frac{1}{2N})\right)\right)}
       {\exp\left(T_{N}\left(\frac{\xi}{2\pi\i}(1-\frac{1}{2N})\right)\right)}
  \\
  &\quad\times
  \sum_{k=0}^{N-1}
  e^{-k\xi}
  \frac{\exp\left(T_{N}\left(\frac{\xi}{2\pi\i}(1-\frac{2k+1}{2N})\right)\right)}
       {\exp\left(T_{N}\left(\frac{\xi}{2\pi\i}(1+\frac{2k+1}{2N})\right)\right)}.
\end{split}
\end{equation*}
We use the following lemma, a proof of which is given in Section~\ref{sec:proofs_lemmas}.
\begin{lem}\label{lem:T_special}
We have
\begin{equation*}
  \frac{\exp\left(T_{N}\left(\frac{\xi}{2\pi\i}(1+\frac{1}{2N})\right)\right)}
       {\exp\left(T_{N}\left(\frac{\xi}{2\pi\i}(1-\frac{1}{2N})\right)\right)}
  =
  \frac{1}{1-e^{\xi}}
  =
  \frac{-e^{-\xi/2}}{2\sinh(\xi/2)}.
\end{equation*}
\end{lem}
\par
Now we define
\begin{equation*}
  f_{N}(z)
  :=
  \frac{1}{N}
  \left(
    T_{N}\left(\frac{\xi(1-z)}{2\pi\i}\right)
    -
    T_{N}\left(\frac{\xi(1+z)}{2\pi\i}\right)
  \right)
  -\xi z
  +2\pi\i z
\end{equation*}
so that
\begin{equation}\label{eq:Jones_sum}
  J_N\bigl(E;\exp(\xi/N)\bigr)
  =
  \frac{1}{2\sinh(\xi/2)}
  \sum_{k=0}^{N-1}\exp\left(N\times f_{N}\left(\frac{2k+1}{2N}\right)\right).
\end{equation}
Since $T_{N}(z)$ is defined for $z$ with $0<\Re(z e^{\theta\i})<\cos\theta$, the function $f_{N}(z)$ is defined for $z$ with
\begin{equation*}
  \left|\frac{\Im{z}}{\tan\theta}+\Re{z}\right|<1
\end{equation*}
from the assumption $\tan\theta<\frac{\pi}{\xi}$.
\par
From Proposition~\ref{prop:T_N_L_2}, $f_N(z)$ converges to the following function:
\begin{equation*}
  F(z)
  :=
  \frac{1}{\xi}
  \left(
    \L_2\left(\frac{\xi(1-z)}{2\pi\i}\right)
    -
    \L_2\left(\frac{\xi(1+z)}{2\pi\i}\right)
  \right)
  -\xi z+2\pi\i z
\end{equation*}
in the region
\begin{equation*}
  \left\{
    z\in\C\Bigm|
    \left|\frac{\Im{z}}{\tan\theta}+\Re{z}\right|
    \le
    1-\frac{2\pi\nu}{\xi\sin\theta},
    |\Re{z}|\le\frac{2\pi M}{\xi}-1
  \right\}.
\end{equation*}
Since $\Im\left(\frac{\xi(1\pm z)}{2\pi\i}\right)=\frac{-\xi}{2\pi}(1\pm\Re{z})$,
\begin{equation*}
  F(z)
  =
  \frac{1}{\xi}
  \left(
    \Li_2\left(e^{-\xi(1+z)}\right)
    -
    \Li_2\left(e^{-\xi(1-z)}\right)
  \right)
  +\xi z
\end{equation*}
when $-1<\Re{z}<1$ from \eqref{eq:L2}.
\begin{rem}\label{rem:F_real}
When $x$ is real and $-1<x<1$, then $F(x)$ is also real.
\end{rem}
From Lemma~\ref{lem:der_L}, the first and the second derivatives of $F(z)$ are given as follows:
\begin{align*}
  F'(z)
  &=
  \L_1\left(\frac{\xi(1-z)}{2\pi\i}\right)
  +
  \L_1\left(\frac{\xi(1+z)}{2\pi\i}\right)
  -\xi+2\pi\i,
  \\
  F''(z)
  &=
  \frac{\xi}{2\pi\i}
  \left(
    \L_0\left(\frac{\xi(1-z)}{2\pi\i}\right)
    -
    \L_0\left(\frac{\xi(1+z)}{2\pi\i}\right)
  \right)
  \\
  &=
  \frac{\xi(e^{-\xi z}-e^{\xi z})}{e^{\xi}+e^{-\xi}-e^{\xi z}-e^{-\xi z}}.
\end{align*}
Note that when $-1<\Re{z}<1$, we have
\begin{equation}\label{eq:F'}
\begin{split}
  F'(z)
  &=
  \log(1-e^{-\xi(1-z)})+\log(1-e^{-\xi(1+z)})+\xi
  \\
  &=
  \log(e^{\xi}+e^{-\xi}-e^{\xi z}-e^{-\xi z})
\end{split}
\end{equation}
from  \eqref{eq:L1}, since $\Im\left(\frac{\xi(1\pm z)}{2\pi\i}\right)=-\frac{\xi}{2\pi}(1\pm\Re{z})$.
\par
We assume that $\xi>\kappa$, where $\kappa:=\arccosh(3/2)=\log(\frac{3+\sqrt{5}}{2})=0.962\dots$.
Put
\begin{equation*}
  \varphi(\xi)
  :=
  \arccosh\left(\cosh(\xi)-\frac{1}{2}\right)
  =
  \log
  \left(
    \cosh(\xi)-\frac{1}{2}+\frac{1}{2}\sqrt{(2\cosh(\xi)-1)^2-4}
  \right).
\end{equation*}
Note that since $\xi>\kappa$, we have $\cosh(\xi)>\cosh(\varphi(\xi))$ and $\cosh(\xi)>\cosh(\kappa)=\frac{3}{2}$.
So we have $0<\varphi(\xi)<\xi$.
Note also that this definition is the same as that in Section~\ref{sec:introduction}.
\par
Since $0<\varphi(\xi)<\xi$, we have
\begin{align}
  F(\varphi(\xi)/\xi)
  &=
  \frac{1}{\xi}
  \left(
    \Li_2(e^{-\xi-\varphi(\xi)})-\Li_2(e^{-\xi+\varphi(\xi)})
  \right)
  +\varphi(\xi),
  \label{eq:F_phi}
  \\
  F'(\varphi(\xi)/\xi)
  &=0,
  \label{eq:F'_phi}
  \\
  F''(\varphi(\xi)/\xi)
  &=
  -\xi\sqrt{(2\cosh(\xi)-1)^2-4}.
  \label{eq:F''_phi}
\end{align}
The second equality follows since
\begin{equation*}
  e^{\varphi(\xi)}+e^{-\varphi(\xi)}
  =
  e^{\xi}+e^{-\xi}-1.
\end{equation*}
\par
Put
\begin{equation*}
  S(\xi)
  :=
  \xi F(\varphi(\xi)/\xi)
  =
  \Li_2(e^{-\xi-\varphi(\xi)})-\Li_2(e^{-\xi+\varphi(\xi)})
  +
  \xi\varphi(\xi).
\end{equation*}
Since
\begin{equation*}
\begin{split}
  \frac{d\,S(\xi)}{d\,\xi}
  &=
  (1+\varphi'(\xi))\log(1-e^{-\xi-\varphi(\xi)})
  -
  (1-\varphi'(\xi))\log(1-e^{-\xi+\varphi(\xi)})
  \\
  &\quad
  +\varphi(\xi)+\xi\varphi'(\xi)
  \\
  &=
  \log\frac{e^{\varphi(\xi)}-e^{-\xi}}{1-e^{-\xi+\varphi(\xi)}}
  +
  \varphi'(\xi)\log\left(e^{\xi}+e^{-\xi}-e^{\varphi(\xi)}-e^{-\varphi(\xi)}\right)
  \\
  &=
  \log\frac{e^{\varphi(\xi)}-e^{-\xi}}{1-e^{-\xi+\varphi(\xi)}}
  >0
\end{split}
\end{equation*}
and $S(\kappa)=0$, $S(\xi)>0$ for $\xi>\kappa$.
\par
Recalling that $F(x)\in\R$ when $x$ is real and $-1<x<1$, we have the following lemma.
\begin{lem}
As a function of a real variable $x$, the real-valued function $F(x)$ takes its maximum $S(\xi)/\xi$ at $x=\varphi(\xi)/\xi$ in the half open interval $[0,1)$.
Moreover, it is strictly increasing {\rm(}decreasing, respectively{\rm)} when $x<\varphi(\xi)/\xi$ {\rm(}$x>\varphi(\xi)/\xi$, respectively{\rm)}.
\end{lem}
\begin{proof}
Since we have
\begin{equation*}
  F'(x)
  =
  \log(e^{\xi}+e^{-\xi}-e^{\xi x}-e^{-\xi x}),
\end{equation*}
we see that $F'(x)$ is strictly decreasing.
Since $F(0)=0$ and $F'(x)$ becomes $0$ when $x=\varphi(\xi)/\xi<1$, the lemma follows.
\end{proof}
We have the following corollary.
\begin{cor}
For any $c_{-}$ and $c_{+}$ with $0<c_{-}<\varphi(\xi)/\xi<c_{+}<1$, there exists $\delta>0$ such that $F(x)<S(\xi)/\xi-\delta$ if $x<c_{-}$ or $x>c_{+}$.
\end{cor}
So we have $e^{N\times F(x)}<e^{N(S(\xi)/\xi-\delta)}$ if $x\not\in[c_{-},c_{+}]$.
Since $f_N(x)$ uniformly converges to $F(x)$ from Proposition~\ref{prop:T_N_L_2}, we also have $\left|e^{N\times f_N(x)}\right|<e^{N(S(\xi)/\xi-\delta')}$ for some $\delta'>0$ if $x\not\in[c_{-},c_{+}]$ and $N$ is sufficiently large.
\par
Therefore we have
\begin{align*}
  \left|\sum_{0<k/N<c_{-}}\exp\left(N\times f_N\left(\frac{2k+1}{2N}\right)\right)\right|
  &<
  c_{-}Ne^{N(S(\xi)/\xi-\delta')}
  \\
  \intertext{and}
  \left|\sum_{c_{+}<k/N<1}\exp\left(N\times f_N\left(\frac{2k+1}{2N}\right)\right)\right|
  &<
  (1-c_{+})Ne^{N(S(\xi)/\xi-\delta')}.
\end{align*}
As a result, we have
\begin{equation*}
\begin{split}
  &\sum_{k=0}^{N-1}\exp\left(N\times f_N\left(\frac{2k+1}{2N}\right)\right)
  \\
  =&
  \sum_{c_{-}\le k/N\le c_{+}}\exp\left(N\times f_N\left(\frac{2k+1}{2N}\right)\right)
  +
  O\left(Ne^{N(S(\xi)/\xi-\delta')}\right).
\end{split}
\end{equation*}
From \eqref{eq:Jones_sum}, we also have
\begin{multline}\label{eq:J_summation}
  J_N\bigl(E;\exp(\xi/N)\bigr)
  \\
  =
  \frac{1}{2\sinh(\xi/2)}
  \sum_{c_{-}\le k/N\le c_{+}}\exp\left(N\times f_N\left(\frac{2k+1}{2N}\right)\right)
  +
  O\left(Ne^{N(S(\xi)/\xi-\delta')}\right).
\end{multline}
\section{Integration}\label{sec:integration}
In this section, we use the Poisson summation formula (see \cite[Proposition~4.2]{Ohtsuki:QT2016}) to change the summation in \eqref{eq:J_summation} into an integration.
Then by using the saddle point method (see also \cite[Proposition~3.2]{Ohtsuki:QT2016}) to prove Theorem~\ref{thm:main}.
\par
Define
\begin{equation*}
  \psi(z)
  :=
  F(z+\varphi(\xi)/\xi)-F(\varphi(\xi)/\xi)
\end{equation*}
in the region
\begin{equation}\label{eq:region_psi}
  \left\{
    z\in\C
    \Bigm|
    -\frac{\varphi(\xi)}{\xi}\le\Re{z}<1-\frac{\varphi(\xi)}{\xi},
    -1-\frac{\varphi(\xi)}{\xi}<\frac{\Im{z}}{\tan\theta}+\Re{z}<1-\frac{\varphi(\xi)}{\xi}
  \right\}.
\end{equation}
Then since $\psi(0)=0$, $\psi'(0)=0$, and $\psi''(0)=-\xi\sqrt{(2\cosh(\xi)-1)^2-4}$ from \eqref{eq:F'_phi} and \eqref{eq:F''_phi}, $\psi(z)$ is of the form
\begin{equation*}
  \psi(z)
  =
  \frac{-\xi}{2}\sqrt{(2\cosh(\xi)-1)^2-4}\times z^2+a_3z^3+a_4z^4+\cdots
\end{equation*}
in \eqref{eq:region_psi}.
\begin{lem}\label{lem:psi_negative}
If $x\ne0$ is real and satisfies $-\varphi(\xi)/\xi\le x<1-\varphi(\xi)/\xi$, then $\psi(x)<0$.
\end{lem}
\begin{proof}
Since $e^{\xi}+e^{-\xi}-e^{\varphi(\xi)}-e^{-\varphi(\xi)}=1$, we have
\begin{equation}\label{eq:psi'}
\begin{split}
  \psi'(x)
  &=
  \log(e^{\xi}+e^{-\xi}-e^{\xi x+\varphi(\xi)}-e^{-\xi x-\varphi(\xi)})
  \\
  &=
  \log(e^{\varphi(\xi)}+e^{-\varphi(\xi)}-e^{\xi x+\varphi(\xi)}-e^{-\xi x-\varphi(\xi)}+1)
  \\
  &=
  \log
  \left(
    e^{-\varphi(\xi)}(1-e^{\xi x})(e^{2\varphi(\xi)}-e^{-\xi x})+1
  \right)
\end{split}
\end{equation}
from \eqref{eq:F'}.
Therefore we see that $\psi'(x)=0$ if $x=0$ or $-2\varphi(\xi)/\xi$, $\psi'(x)>0$ if $-2\varphi(\xi)/\xi<x<0$, and $\psi'(x)<0$ otherwise.
\par
Therefore for $-\varphi(\xi)/\xi<x<0$, $\psi(x)$ is monotonically increasing and for $0<x<1-\varphi(\xi)/\xi$ it is monotonically decreasing.
So $\psi(x)$ takes its unique maximum $0$ at $x=0$, which shows that $\psi(x)<0$ when $x\ne0$.

\end{proof}
Now we use the following proposition.
\begin{prop}[{\cite[Proposition~4.2]{Ohtsuki:QT2016}}]\label{prop:Poisson}
Let $b_{-}$ and $b_{+}$ be real numbers with $b_{-}<0<b_{+}$.
Put
\begin{align*}
  \Lambda
  &:=
  \left\{
    \frac{k}{N}\Bigm| k\in\Z, b_{-}\le\frac{k}{N}\le b_{+}
  \right\},
  \\
  C
  &:=
  \{t\in\R\mid b_{-}\le t\le b_{+}\},
  \\
  D
  &:=
  \{z\in\C\mid\Re{\psi(z)}<0\}.
\end{align*}
Assume that $\psi(z)$ is a holomorphic function of the form
\begin{equation}\label{eq:f}
  \psi(z)
  =
  az^2+a_3z^3+a_4z^4+\cdots
\end{equation}
with $\Re(a)<0$, defined in a neighborhood $P$ of $0\in\C$ that includes the $\delta_0$-neighborhood $N_{\delta_0}$ of $0$ for $\delta_0>0$.
We choose $P$ so that the region $D\cap P$ has two connected components.
We also assume the following:
\begin{enumerate}
\item
$b_{-}$ and $b_{+}$ are in different components of $D\cap P$ and moreover $\Re{\psi(b_{\pm})}<-\varepsilon_0$ for some $\varepsilon_0>0$,
\item
Both $b_{-}$ and $b_{+}$ are in a connected component of
\begin{equation}\label{eq:R+}
  \{x+y\i\mid x\in[b_{-},b_{+}],y\in[0,\delta_0],\Re{\psi(x+y\i)}<2\pi y\}
\end{equation}
\item
Both $b_{-}$ and $b_{+}$ are in a connected component of
\begin{equation}\label{eq:R-}
  \{x-y\i\mid x\in[b_{-},b_{+}],y\in[0,\delta_0],\Re{\psi(x-y\i)}<2\pi y\}
\end{equation}
\end{enumerate}
Then there exists $\varepsilon>0$, depending on $a$ and $\varepsilon_0$, such that
\begin{equation*}
  \frac{1}{N}
  \sum_{z\in\Lambda}
  e^{N\psi(z)}
  =
  \int_{C}e^{N\psi(z)}\,dz
  +
  O(e^{-N\varepsilon}).
\end{equation*}
\end{prop}
We will show that $\psi(z)$ satisfies the assumptions of Proposition~\ref{prop:Poisson}.
\par
Put $b_{-}:=-\varphi(\xi)/(2\xi)$, $b_{+}:=(1-\varphi(\xi)/\xi)/2$, $\delta_0:=\frac{1}{2}(1-\varphi(\xi)/\xi)\sin\theta$.
Define the following regions:
\begin{align*}
  P
  &:=
  \left\{
    z\in\C
    \mid
    -\frac{\varphi(\xi)}{\xi}<\Re{z}<1-\frac{\varphi(\xi)}{\xi},
    -1-\frac{\varphi(\xi)}{\xi}<\frac{\Im{z}}{\tan\theta}+\Re{z}<1-\frac{\varphi(\xi)}{\xi}
  \right\},
  \\
  D_{-}
  &:=
  \{z\in\C\mid\Re{\psi(z)}<0,\Re{z}<0\},
  \\
  D_{+}
  &:=
  \{z\in\C\mid\Re{\psi(z)}<0,\Re{z}>0\}.
\end{align*}
Note that $P$ is just the region \eqref{eq:region_psi} and so $\psi(z)$ is holomorphic in $P$.
See Figure~\ref{fig:P}.
\begin{figure}[h!]
\includegraphics{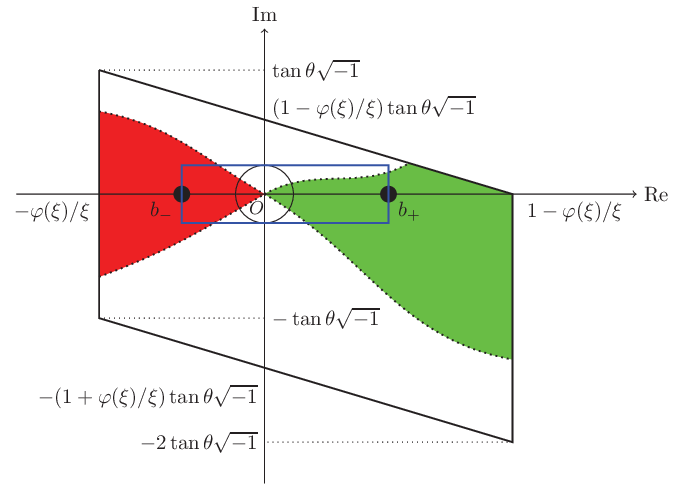}
\caption{The parallelogram is the boundary of $P$, the circle is the boundary of $N_{\delta_0}$, the red region is $D_{-}$, and the green region is $D_{+}$.
The cyan rectangle indicates the boundary of $R_{+}\cup R_{-}$} (see below).
\label{fig:P}
\end{figure}
\par
Then, we show the following lemma, from which the assumptions of Proposition~\ref{prop:Poisson} hold.
\begin{lem}\label{lem:Poisson_psi}
We assume that $\sin{\theta}<\frac{\varphi(\xi)}{\xi-\varphi(\xi)}$ and that $\tan\theta<\frac{\pi}{2\xi}$.
Then, the function $\psi(z)$ satisfies the following:
\begin{enumerate}
\item[(i).]
$\psi(z)$ is a holomorphic function of the form \eqref{eq:f} defined in the $\delta_0$-neighborhood $N_{\delta_0}$ of $0\in\C$.
\item[(ii).]
Both $D_{+}\cap P$ and $D_{-}\cap P$ are connected, and $b_{\pm}\in D_{\pm}\cap P$.
Moreover, $\Re\psi(b_{\pm})<-\varepsilon_0$ for some $\varepsilon_0>0$.
\item[(iii).]
Condition (2) in Proposition~\ref{prop:Poisson} is satisfied.
\item[(iv).]
Condition (3) in Proposition~\ref{prop:Poisson} is satisfied.
\end{enumerate}
\end{lem}
\begin{proof}
\quad\par
\begin{enumerate}
\item[(i).]
The distances to the upper side, the lower side, the left side, and the right side of $P$ from the origin $O$ is $(1-\varphi(\xi)/\xi)\sin\theta$, $(1+\varphi(\xi)/\xi)\sin\theta$, $\varphi(\xi)/\xi$ and $1-\varphi(\xi)/\xi$, respectively.
Since $(1-\varphi(\xi)/\xi)\sin\theta<\varphi(\xi)/\xi$ from the assumption $\sin\theta<\frac{\varphi}{\xi-\varphi(\xi)}$, $\delta_0$ is less than these distances.
So $N_{\delta_0}$ is contained in the parallelogram $P$.
Moreover since $\psi''(0)=F''(\varphi(\xi)/\xi)=-\xi\sqrt{(2\cosh(\xi)-1)^2-4}<0$, $\psi(z)$ is of the form \eqref{eq:f} with $a:=-\frac{\xi}{2}\sqrt{(2\cosh(\xi)-1)^2-4}$.
\item[(ii).]
We first show that $b_{\pm}\in D_{\pm}\cap P$.
It is clear that $b_{-}<0$, $b_{+}>0$, and $b_{\pm}\in P$.
The inequality $\Re\psi(b_{\pm})<0$ follows from Lemma~\ref{lem:psi_negative}.
So, we can choose $\varepsilon_0>0$ so that $\Re\psi(b_{\pm})<-\varepsilon_0$.
\par
Next, we will show that for each $x\ne0$ with $-\varphi(\xi)/\xi<x<1-\varphi(\xi)/\xi$, the set $\{y\in\R\mid\Re\psi(x+y\i)<0\}\cap P$ is an open interval containing $0$.
Then, we can see that $D_{\pm}\cap P$ is connected.
\par
If we put
\begin{equation*}
  g(x,y)
  :=
  e^{\xi}+e^{-\xi}-e^{\xi(x+y\i)+\varphi(\xi)}-e^{-\xi(x+y\i)-\varphi(\xi)},
\end{equation*}
then we have
\begin{equation}\label{eq:arg}
  \frac{d}{d\,y}\Re\psi(x+y\i)
  =
  -\arg g(x,y)
\end{equation}
from \eqref{eq:psi'}.
Now we have
\begin{align*}
  \Re g(x,y)
  &=
  2\cosh{\xi}
  -
  2\cosh\left(\xi\left(x+\frac{\varphi(\xi)}{\xi}\right)\right)\cos(\xi y),
  \intertext{and}
  \Im g(x,y)
  &=
  -\sinh\left(\xi\left(x+\frac{\varphi(\xi)}{\xi}\right)\right)\sin(\xi y).
  \notag
\end{align*}
\par
We can see that $\Re g(x,y)$ is positive since $0<x+\varphi(\xi)/\xi<1$.
We will show that $\Im g(x,y)<0$ ($\Im g(x,y)>0$) if $y>0$ ($y<0$, respectively).
If $x+y\i\in P$, then $-2\tan\theta<y<\tan\theta$ (see Figure~\ref{fig:P}).
Together with the assumption $\tan\theta<\frac{\pi}{2\xi}$, we conlude that $|y|\le2\tan\theta<\pi/\xi$, and so $|\xi y|<\pi$.
This implies that $\sin(\xi y)>0$ ($\sin(\xi y)<0$, respectively) if $y>0$ ($y<0$, respectively).
Since $x+\varphi(\xi)/\xi>0$, we see that $\Im g(x,y)<0$ ($\Im g(x,y)>0$) if $y>0$ ($y<0$, respectively).
\par
Therefore $\arg{g(x,y)}$ is positive (negative, respectively) when $y<0$ ($y>0$, respectively).
It follows that if $y>0$ ($y<0$, respectively), $\Re\psi(x+y\i)$ is monotonically increasing (decreasing, respectively) with respect to $y$.
Since $\Re\psi(x)<0$ from Lemma~\ref{lem:psi_negative}, we see that $\{y\in\R\mid\Re\psi(x+y\i)<0\}\cap P$ is an open interval containing $0$.
\par
Therefore, we conclude that $D_{\pm}\cap P$ is connected.
\item[(iii).]
We will show that the region $R_{+}:=[b_{-},b_{+}]\times[0,\delta_0]\setminus\{O\}$ is contained in the region $\{x+y\i\in\C\mid\Re\psi(x+y\i)<2\pi y\}$, which means that the region \eqref{eq:R+} equals $R_{+}$.
Then we can find a path connecting $b_{-}$ and $b_{+}$ in $R_{+}$, proving Condition (2).
\par
In fact, we can show if $b_{-}\le x\le b_{+}$ and $0\le y\le\delta_0$ ($(x,y)\ne(0,0)$), then $\Re\psi(x+y\i)-2\pi y<0$ as follows.
\par
When $y>0$, then $\Im g(x,y)<0$ from the argument in (ii), and so $g(x,y)$ is in the fourth quadrant in the complex plane because $\Re g(x,y)>0$.
Therefore we have $0>\arg g(x,y)>-\pi/2$.
This shows that $\frac{d}{d\,y}\bigl(\Re\psi(x+y\i)-2\pi y\bigr)<0$ when $y>0$ from \eqref{eq:arg}, and so we conclude that $\Re\psi(x+y\i)-2\pi y<0$ since $\Re\psi(x)\le0$ from Lemma~\ref{lem:psi_negative}.
\item[(iv).]
We will show that the region $R_{-}:=[b_{-},b_{+}]\times[-\delta_0,0]\setminus\{O\}$ is in the region $\{x+y\i\in\C\mid\Re\psi(x+y\i)<-2\pi y\}$, which means that the region \eqref{eq:R-} equals $R_{-}\cup\{O\}$.
Then we can find a path connecting $b_{-}$ and $b_{+}$ in $R_{-}$, proving Condition (3).
\par
We will show if $b_{-}\le x\le b_{+}$ and $-\delta_0\le y\le0$ ($(x,y)\ne(0,0)$), then $\Re\psi(x+y\i)+2\pi y<0$.
\par
When $y<0$, then $\Im g(x,y)>0$ from the argument in (ii), and so $g(x,y)$ is in the first quadrant because $\Re g(x,y)>0$.
Therefore we have $0<\arg g(x,y)<\pi/2$.
This shows that $\frac{d}{d\,y}\bigl(\Re\psi(x+y\i)+2\pi y\bigr)>0$ when $y<0$, and so we conclude that $\Re\psi(x+y\i)+2\pi y<0$ since $\Re\psi(x)\le0$ from Lemma~\ref{lem:psi_negative}.
\end{enumerate}
\end{proof}
Now we can prove Theorem~\ref{thm:main}.
\begin{proof}[Proof of Theorem~\ref{thm:main}]
From Proposition~\ref{prop:Poisson}, there exists $\varepsilon>0$ such that
\begin{equation*}
  \frac{1}{N}
  \sum_{-\varphi(\xi)/(2\xi)\le k/N\le(1-\varphi(\xi)/\xi)/2}
  e^{N\psi(k/N)}
  \\
  =
  \int_{-\varphi(\xi)/(2\xi)}^{(1-\varphi(\xi)/\xi)/2}
  e^{N\psi(z)}\,dz
  +
  O\left(e^{-N\varepsilon}\right).
\end{equation*}
Since $f_N(z+\varphi(\xi)/\xi+1/(2N))-S(\xi)/\xi$ uniformly converges to $\psi(z)$, we have
\begin{multline*}
  \frac{1}{N}
  \sum_{-\varphi(\xi)/(2\xi)\le k/N\le(1-\varphi(\xi)/\xi)/2}
  e^{N\left(f_N(k/N+\varphi(\xi)/\xi+1/(2N))-S(\xi)/\xi\right)}
  \\
  =
  \int_{-\varphi(\xi)/(2\xi)}^{(1-\varphi(\xi)/\xi)/2}
  e^{N\psi(z)}\,dz
  +O(e^{-N\varepsilon})
\end{multline*}
from \cite[Remark~4.4]{Ohtsuki:QT2016}.
Putting $l/N:=k/N+\varphi(\xi)/\xi$, we also have
\begin{multline*}
  \sum_{-\varphi(\xi)/(2\xi)\le k/N\le(1-\varphi(\xi)/\xi)/2}
  e^{N\left(f_N\left(k/N+\varphi(\xi)/\xi+1/(2N)\right)-S(\xi)/\xi\right)}
  \\
  =
  e^{-N S(\xi)/\xi}
  \sum_{\varphi(\xi)/(2\xi)\le l/N\le(1+\varphi(\xi)/\xi)/2}
  \exp\left(N\times f_N\left(\frac{2l+1}{2N}\right)\right).
\end{multline*}
Thus we obtain
\begin{multline}\label{eq:Poisson_f_N}
  \sum_{\varphi(\xi)/(2\xi)\le k/N\le(1+\varphi(\xi)/\xi)/2}
  \exp\left(N\times f_N\left(\frac{2k+1}{2N}\right)\right)
  \\
  =
  N\times e^{N S(\xi)/\xi}
  \left(
    \int_{-\varphi(\xi)/(2\xi)}^{(1-\varphi(\xi)/\xi)/2}
    e^{N\psi(z)}\,dz
    +
    O\left(e^{-N\varepsilon}\right)
  \right).
\end{multline}
Now by using the saddle point method \cite[Proposition~3.2]{Ohtsuki:QT2016} (see also \cite[Remark~3.3]{Ohtsuki:QT2016}), we have
\begin{equation}\label{eq:saddle}
  \int_{-\varphi(\xi)/(2\xi)}^{(1-\varphi(\xi)/\xi)/2}
  e^{N\psi(z)}\,dz
  =
  \frac{\sqrt{\pi}}{\sqrt{\frac{\xi}{2}\sqrt{(2\cosh(\xi)-1)^2-4}}\sqrt{N}}
  \left(
    1+O(N^{-1})
  \right).
\end{equation}
Putting $c_{-}:=\varphi(\xi)/(2\xi)$ and $c_{+}:=(1+\varphi(\xi)/\xi)/2$ in \eqref{eq:J_summation}, we finally have
\begin{equation*}
\begin{split}
  &J_N\bigl(E;\exp(\xi/N)\bigr)
  \\
  =&
  \frac{N\times e^{N S(\xi)/\xi}}{2\sinh(\xi/2)}
  \left(
    \int_{-\varphi(\xi)/(2\xi)}^{(1-\varphi(\xi)/\xi)/2}
    e^{N\psi(z)}\,dz
    +O\left(e^{-N\min\{\varepsilon,\delta'\}}\right)
  \right)
  \\
  =&
  \frac{\sqrt{\pi}\sqrt{N} e^{N\times S(\xi)/\xi}}
       {2\sinh(\xi/2)\sqrt{\frac{\xi}{2}\sqrt{(2\cosh(\xi)-1)^2-4}}}
  \left(1+O(N^{-1})\right).
\end{split}
\end{equation*}
where we use \eqref{eq:Poisson_f_N} at the first equality and \eqref{eq:saddle} at the second.
Letting $T(\xi)$ denote $\frac{2}{\sqrt{(2\cosh(\xi)-1)^2-4}}$, we have
\begin{equation*}
  J_N\bigl(E;\exp(\xi/N)\bigr)
  =
  \frac{\sqrt{\pi}}{2\sinh(\xi/2)}
  \sqrt{T(\xi)}
  \sqrt{\frac{N}{\xi}}
  \times
  e^{\frac{N}{\xi}S(\xi)}
  \left(1+O(N^{-1})\right).
\end{equation*}
Since $J_{N}(E;e^{-\xi/N})=J_{N}(E;e^{\xi/N})$ we obtain the required formula.
\end{proof}
\section{Topological interpretations}\label{sec:CST}
In this section we give topological interpretations for $T(\xi)$ and $S(\xi)$.
\subsection{Representations}
Let $X$ be the complement of the open tubular neighborhood of $E\subset S^3$.
The fundamental group $\pi_1(X)$ is presented as
\begin{equation*}
  \langle
    x,y
    \mid
    wx=yw
  \rangle,
\end{equation*}
with $w:=xy^{-1}x^{-1}y$.
Let $\rho$ be a non-Abelian representation of $\pi_1(X)$ to $\SL(2;\C)$ given by
\begin{equation}\label{eq:presentation_F_2}
  \rho(x)
  =
  \begin{pmatrix}e^{\xi/2}&1\\0&e^{-\xi/2}\end{pmatrix},
  \quad
  \rho(y)
  =
  \begin{pmatrix}e^{\xi/2}&0\\-d&e^{-\xi/2}\end{pmatrix},
\end{equation}
where $d$ annihilates the Riley polynomial
\begin{equation}\label{eq:Riley}
  d^2-(2\cosh(\xi)-3)d-2\cosh(\xi)+3.
\end{equation}
The preferred longitude $\lambda$ is presented by $y^{-1}xyx^{-2}yxy^{-1}$ and it is sent to
\begin{equation*}
  \rho(\lambda)
  =
  \begin{pmatrix}
    e^{\eta/2}&\ast \\
    0          &e^{-\eta/2}
  \end{pmatrix},
\end{equation*}
where
\begin{multline*}
  \eta
  :=
  \log
  \left(
    \frac{1}{2}
    \left(\vphantom{\frac{1}{2}}
      e^{2\xi}-e^{\xi}-e^{-\xi}+e^{-2\xi}-2
    \right.
  \right.
  \\
  \left.
    \left.
      +(e^{-\xi}-e^{\xi})\sqrt{(e^{\xi}+e^{-\xi}-3)(e^{\xi}+e^{-\xi}+1)}
    \right)
  \right).
\end{multline*}
\subsection{Adjoint Reidemeister torsion}
For a representation $\rho\colon\pi_1(X)\to\SL(2;\C)$, one can consider the cochain complex $C^{\ast}(X;\mathfrak{sl}(2;\C)_{\rho}):=\Hom_{\Z[\pi_1(X)]}(C_{\ast}(\tilde{X};\Z),\mathfrak{sl}(2;\C))$ twisted by the adjoint action of $\rho$.
Here $\tilde{X}$ is the universal cover of $X$, $\pi_1(X)$ acts on $\tilde{X}$ as the deck transformation, and the Lie algebra $\mathfrak{sl}(2;\C)$ is regarded as a $\Z[\pi_1(X)]$-module by the adjoint action of $\rho$.
The Reidemeister torsion $\Tor_{\mu}(\rho)\in\C$ associated with the meridian $\mu$, twisted by the adjoint action of $\rho$ is defined as the torsion of the cochain complex $C^{\ast}(X;\mathfrak{sl}(2;\C)_{\rho})$.
The following formula is known for the case of the figure-eight knot:
\begin{equation*}
  {\Tor}_{\mu}(\rho)
  =
  \pm
  \frac{2}{\sqrt{(e^{\xi}+e^{-\xi}+1)(e^{\xi}+e^{\xi}-3)}}.
\end{equation*}
See \cite{Porti:MAMCAU1997,Dubois:CANMB2006,Murakami:JTOP2013}.
\subsection{Chern--Simons invariant}
Let $M$ be a three-manifold with boundary a torus $T$, and $\{\mu,\lambda\}$ be generators of $\pi_1(T)$.
For a representation $\rho\colon\pi_1(M)\to\SL(2;\C)$, we can define the Chern--Simons invariant as follows.
\par
Let $A$ be an $\sl(2;\C)$-valued $1$-form $A$ on $M$ that defines the flat connection corresponding to $\rho$.
Assume that $\rho\Bigm|_{T}$ is diagonalizable for simplicity.
Then by a suitable conjugation, one has
\begin{equation*}
  \rho(\mu)
  =
  \begin{pmatrix}
    e^{2\pi\i\alpha}& 0               \\
    0               & e^{-2\pi\i\alpha}
  \end{pmatrix},
  \quad
  \rho(\lambda)
  =
  \begin{pmatrix}
    e^{2\pi\i\beta}& 0               \\
    0              &e^{-2\pi\i\beta}
  \end{pmatrix}.
\end{equation*}
Then up to gauge equivalence, we can assume that $A$ is of the form
\begin{equation*}
  \begin{pmatrix}
    \i\alpha& 0       \\
    0       & -\i\alpha
  \end{pmatrix}
  \,dx
  +
  \begin{pmatrix}
    \i\beta& 0      \\
    0      &-\i\beta
  \end{pmatrix}
  \,dy
\end{equation*}
near $T$, where $dx$ and $dy$ are the $1$-forms corresponding to $\mu$ and $\lambda$ respectively.
Then the Chern--Simons invariant $\cs_{M}(\rho;\alpha,\beta)$ of $\rho$ associated with $(\alpha,\beta)$ is defined by
\begin{equation*}
  \CS_{M}(\rho;\alpha,\beta)
  :=
  \frac{-1}{8}
  \int_{M}\Tr\left(dA\wedge A+\frac{2}{3}A\wedge A\wedge A\right)\in\C/(\pi^2\Z).
\end{equation*}
Note that $\Im\CS_{M}(\rho_0;0,0)$ coincides with the hyperbolic volume if the interior of $M$ possesses a complete hyperbolic structure, where $\rho_0$ is the Levi--Civita connection.
See \cite{Kirk/Klassen:COMMP1993} for details.
\par
In \cite{Murakami/Tran:2020}, we proved the following theorem.
\begin{thm}[{\cite{Murakami/Tran:2020}}]
The Chern--Simons invariant $\CS_{E}(\rho;\xi,\eta)$ of the representation $\rho$ associated with $(\xi,\eta)$ is given by
\begin{equation*}
  \CS_{E}(\rho;\xi,\eta)
  =
  S(\xi)-\frac{\xi\eta}{2}.
\end{equation*}
Here we choose $\mu$ and $\lambda$ to be the meridian and the preferred longitude of $E$ respectively, and we assume that $\rho$ sends
\begin{equation*}
  \mu
  \mapsto
  \begin{pmatrix}
    e^{\xi/2}&\ast \\
    0          &e^{-\xi/2}
  \end{pmatrix},
  \quad
  \lambda
  \mapsto
  \begin{pmatrix}
    e^{\eta/2}&\ast \\
    0         &e^{-\eta/2}
  \end{pmatrix}.
\end{equation*}
up to conjugation.
\end{thm}
\section{Proofs of lemmas}\label{sec:proofs_lemmas}
In this section, we give proofs of lemmas used in this paper.
\begin{proof}[Proof of Lemma~\ref{lem:converge_T}]
Putting $x:=y\,e^{\theta\i}$, the integral becomes
\begin{equation*}
  \int_{C_{0}}
  \frac{\exp\left((2z-1)y\,e^{\theta\i}\right)}
       {y\sinh(y\,e^{\theta\i})\sinh(\gamma y\,e^{\theta\i})}
  \,dy.
\end{equation*}
Noting that
\begin{equation*}
  \sinh(as)
  \underset{s\to\infty}{\sim}
  \frac{1}{2}e^{as}
\end{equation*}
and that
\begin{equation*}
  \sinh(as)
  \underset{s\to-\infty}{\sim}
  -\frac{1}{2}e^{-as}
\end{equation*}
for a complex number $a$ with $\Re(a)>0$, we have
\begin{equation*}
  \frac{\exp\left((2z-1)y\,e^{\theta\i}\right)}
       {y\sinh(y\,e^{\theta\i})\sinh(\gamma y\,e^{\theta\i})}
  \underset{y\to\infty}{\sim}
  \frac{1}{y}
  \exp
  \left(
    (2z-2-\gamma)e^{\theta\i}y
  \right)
\end{equation*}
and
\begin{equation*}
  \frac{\exp\left((2z-1)y\,e^{\theta\i}\right)}
       {y\sinh(y\,e^{\theta\i})\sinh(\gamma y\,e^{\theta\i})}
  \underset{y\to-\infty}{\sim}
  \frac{1}{y}
  \exp
  \left(
    (2z+\gamma)e^{\theta\i}y
  \right)
\end{equation*}
since $\Re\left(e^{\theta\i}\right)=\cos\theta>0$ and $\Re\left(\gamma e^{\theta\i}\right)=\frac{\xi}{2N\pi}\sin\theta>0$.
So we have
\begin{equation*}
  \left|
  \frac{\exp\left((2z-1)y\,e^{\theta\i}\right)}
       {y\sinh(y\,e^{\theta\i})\sinh(\gamma y\,e^{\theta\i})}
  \right|
  <
  \frac{C}{y}
  \exp
  \left(
    \Re\left((2z-2-\gamma)e^{\theta\i}\right)y
  \right)
\end{equation*}
when $y>0$ for a positive constant $C$, and
\begin{equation*}
  \left|
  \frac{\exp\left((2z-1)y\,e^{\theta\i}\right)}
       {y\sinh(y\,e^{\theta\i})\sinh(\gamma y\,e^{\theta\i})}
  \right|
  <
  \frac{C'}{|y|}
  \exp
  \left(
    -\Re\left((2z+\gamma)e^{\theta\i}\right)y
  \right)
\end{equation*}
when $y<0$ for a positive constant $C'$.
Therefore, if $-\frac{\xi\sin\theta}{4N\pi}<\Re(ze^{\theta\i})<\cos\theta+\frac{\xi\sin\theta}{4N\pi}$, then the integrals
\begin{equation*}
  \int_{1}^{\infty}
  \frac{\exp\left((2z-1)y\,e^{\theta\i}\right)}
       {y\sinh(y\,e^{\theta\i})\sinh(\gamma y\,e^{\theta\i})}
  dy
\end{equation*}
and
\begin{equation*}
  \int_{-\infty}^{-1}
  \frac{\exp\left((2z-1)y\,e^{\theta\i}\right)}
       {y\sinh(y\,e^{\theta\i})\sinh(\gamma y\,e^{\theta\i})}
  dy
\end{equation*}
converge and the lemma follows.
\end{proof}
\begin{proof}[Proof of Lemma~\ref{lem:integrals}]
Putting $x=ye^{\i\theta}$, we have
\begin{equation*}
  \int_{C_{\theta}}
  \frac{e^{(2z-1)x}}{x^m\sin{x}}\,dx
  =
  \frac{1}{e^{(m-1)\i\theta}}
  \int_{C_{0}}
  \frac{e^{(2z-1)ye^{\i\theta}}}{y^{m}\sinh\left(ye^{\i\theta}\right)}\,dy.
\end{equation*}
So we need to show that
\begin{equation*}
  \int_{1}^{r}
  \frac{e^{(2z-1)ye^{\i\theta}}}{y^{m}\sinh\left(ye^{\i\theta}\right)}\,dy
\end{equation*}
and
\begin{equation*}
  \int_{-r}^{-1}
  \frac{e^{(2z-1)ye^{\i\theta}}}{y^{m}\sinh\left(ye^{\i\theta}\right)}\,dy
\end{equation*}
converge when $r\to\infty$ for $m=0,1,2$.
\par
We have
\begin{equation*}
\begin{split}
  \left|
    \int_{1}^{r}
    \frac{e^{(2z-1)ye^{\i\theta}}}{y^{m}\sinh\left(ye^{\i\theta}\right)}\,dy
  \right|
  &\le
  \int_{1}^{r}
  \frac{2e^{2y\Re\left(ze^{\i\theta}\right)-y\cos\theta}}
       {y^{m}\left(e^{y\cos\theta}-e^{-y\cos\theta}\right)}\,dy
  \\
  &=
  \int_{1}^{r}
  \frac{2e^{2y\left(\Re(ze^{\i\theta})-\cos\theta\right)}}
       {y^{m}(1-e^{-2y\cos\theta})}\,dy,
\end{split}
\end{equation*}
which converges when $r\to\infty$ since $\Re(ze^{\i\theta})<\cos\theta$.
\par
We also have
\begin{equation*}
\begin{split}
  \left|
    \int_{-r}^{-1}
    \frac{e^{(2z-1)ye^{\i\theta}}}{y^{m}\sinh(ye^{\i\theta})}\,dy
  \right|
  &\le
  \int_{-r}^{-1}
  \frac{2e^{2y\Re\left(ze^{\i\theta}\right)-y\cos\theta}}
       {\left|y^{m}\left(e^{y\cos\theta}-e^{-y\cos\theta}\right)\right|}\,dy
  \\
  &=
  \int_{-r}^{-1}
  \frac{2e^{2y\Re(ze^{\i\theta})}}{\left|y^{m}(e^{2y\cos\theta}-1)\right|}\,dy,
\end{split}
\end{equation*}
which converges when $r\to\infty$ since $\Re(ze^{\i\theta})>0$.
\end{proof}
Now we calculate the integrals in Lemma~\ref{lem:integrals} to prove Lemma~\ref{lem:L0_1_2}.
\par
The following lemma shows \eqref{eq:L0}.
\begin{lem}\label{lem:L0}
If $0<\Re(ze^{\theta\i})<\cos\theta$, then we have
\begin{equation*}
  \int_{C_{\theta}}\frac{e^{(2z-1)x}}{\sinh(x)}\,dx
  =
  \frac{-2\pi\i}{1-e^{-2\pi\i z}}.
\end{equation*}
\end{lem}
\begin{proof}
Put $C_{0}^{r}:=C_{0}\setminus\Bigl((-\infty,-r)\cup(r,\infty)\Bigr)$ for $r>1$, $C_{\theta}^{r}:=e^{\i\theta}C_{0}^{r}$, and $C_{\theta}^{r}+\pi\i:=\{w+\pi\i\mid w\in C_{\theta}^{r}\}$.
We first note that
\begin{equation*}
\begin{split}
  \int_{C_{\theta}^{r}}\frac{e^{(2z-1)x}}{\sinh(x)}\,dx
  &=
  \int_{C_{\theta}^{r}+\pi\i}\frac{e^{(2z-1)(x-\pi\i)}}{\sinh(x-\pi\i)}\,d(x-\pi\i)
  \\
  &=
  e^{-2\pi\i z}
  \int_{C_{\theta}^{r}+\pi\i}\frac{e^{(2z-1)x}}{\sinh(x)}\,dx.
\end{split}
\end{equation*}
Hence we have
\begin{equation*}
\begin{split}
  &
  \left(1-e^{-2\pi\i z}\right)
  \int_{C_{\theta}^{r}}\frac{e^{(2z-1)x}}{\sinh(x)}\,dx
  \\
  =&
  e^{-2\pi\i z}
  \left(
    -\int_{C_{\theta}^{r}}\frac{e^{(2z-1)x}}{\sinh(x)}\,dx
    +
    \int_{C_{\theta}^{r}+\pi\i}\frac{e^{(2z-1)x}}{\sinh(x)}\,dx
  \right)
  \\
  =&
  e^{-2\pi\i z}
  \left(
    \int_{V_{r}^{+}}\frac{e^{(2z-1)x}}{\sinh(x)}\,dx
    -
    \int_{V_{r}^{-}}\frac{e^{(2z-1)x}}{\sinh(x)}\,dx
  \right)
  \\
  &-
  e^{-2\pi\i z}2\pi\i\Res\left(\frac{e^{(2z-1)x}}{\sinh(x)};x=\pi\i\right),
\end{split}
\end{equation*}
where $V_{r}^{\pm}$ is the vertical segment connecting $\pm re^{\i\theta}$ and $\pm re^{\i\theta}+\pi\i$, oriented upward.
Since $\Res\left(\frac{e^{(2z-1)x}}{\sinh(x)};x=\pi\i\right)=\lim_{x\to\pi\i}\frac{e^{(2z-1)x}(x-\pi\i)}{\sinh(x)}=-e^{(2z-1)\pi\i}$, we have
\begin{equation*}
\begin{split}
  &
  \left(1-e^{-2\pi\i z}\right)
  \int_{C_{\theta}^{r}}\frac{e^{(2z-1)x}}{\sinh(x)}\,dx
  \\
  =&
  e^{-2\pi\i z}
  \left(
    \int_{V_{r}^{+}}\frac{e^{(2z-1)x}}{\sinh(x)}\,dx
    -
    \int_{V_{r}^{-}}\frac{e^{(2z-1)x}}{\sinh(x)}\,dx
  \right)
  -
  2\pi\i.
\end{split}
\end{equation*}
\par
We will show that $\displaystyle\lim_{r\to\infty}\int_{V_{r}^{\pm}}\frac{e^{(2z-1)x}}{\sinh(x)}\,dx=0$.
\par
Since
\begin{equation*}
  \int_{V_{r}^{\pm}}\frac{e^{(2z-1)x}}{\sinh(x)}\,dx
  =
  \pi\i
  \int_{0}^{1}\frac{e^{(2z-1)(\pm re^{\i\theta}+\pi\i s)}}{\sinh(\pm re^{\i\theta}+\pi\i s)}\,ds
\end{equation*}
and $|\sinh(w)|=\frac{1}{2}\left|e^{w}-e^{-w}\right|\ge\frac{1}{2}\left|e^{\Re{w}}-e^{-\Re{w}}\right|$ for any $w\in\C$, we have
\begin{equation*}
\begin{split}
  \left|
    \int_{V_{r}^{\pm}}\frac{e^{(2z-1)x}}{\sinh(x)}\,dx
  \right|
  &\le
  \pi
  \int_{0}^{1}
  \left|
    \frac{e^{(2z-1)(\pm re^{\i\theta}+\pi\i s)}}{\sinh(\pm re^{\i\theta}+\pi\i s)}
  \right|
  \,ds
  \\
  &\le
  \pi
  \int_{0}^{1}
  \frac{2e^{\Re((2z-1)(\pm re^{\i\theta}+\pi\i s))}}
       {\left|e^{\Re(\pm re^{\i\theta}+\pi\i s)}-e^{-\Re(\pm re^{\i\theta}+\pi\i s)}\right|}
  \,ds
  \\
  &=
  \frac{2\pi e^{\pm r\Re((2z-1)e^{\i\theta})}}{e^{r\cos\theta}-e^{-r\cos\theta}}
  \int_{0}^{1}e^{-\pi s\Im(2z-1)}\,ds.
\end{split}
\end{equation*}
From the assumption $0<\Re(ze^{\theta\i})<\cos\theta$, we have $\left|\Re\left((2z-1)e^{\i\theta}\right)\right|<\cos\theta$.
Therefore we see that
\begin{equation*}
  \left|
    \int_{V_{r}^{\pm}}\frac{e^{(2z-1)x}}{\sinh(x)}\,dx
  \right|
  \underset{r\to\infty}{\to}0
\end{equation*}
and so we have
\begin{equation*}
  \int_{C_{\theta}}\frac{e^{(2z-1)x}}{\sinh(x)}\,dx
  =
  \lim_{r\to\infty}\int_{C_{\theta}^{r}}\frac{e^{(2z-1)x}}{\sinh(x)}\,dx
  =
  \frac{-2\pi\i}{1-e^{-2\pi\i z}}.
\end{equation*}
\end{proof}
\par
The following lemma shows \eqref{eq:L1} and \eqref{eq:L2}.
\begin{lem}\label{lem:L1_2}
If $0<\Re(ze^{\theta\i})<\cos\theta$, then we have
\begin{equation*}
  \int_{C_{\theta}}\frac{e^{(2z-1)x}}{x\sinh(x)}\,dx
  =
  \begin{cases}
    -2\log\left(1-e^{2\pi\i z}\right)&\text{if $\Im{z}\ge0$,}
    \\[3mm]
    -2\pi\i(2z-1)-2\log\left(1-e^{-2\pi\i z}\right)&\text{if $\Im{z}<0$.}
  \end{cases}
\end{equation*}
and
\begin{multline*}
  \int_{C_{\theta}}\frac{e^{(2z-1)x}}{x^2\sinh(x)}\,dx
  \\
  =
  \begin{cases}
    -\dfrac{2\i}{\pi}\Li_2\left(e^{2\pi\i z}\right)&\text{if $\Im{z}\ge0$,}
    \\[3mm]
    -2\pi\i\left(2z^2-2z+\frac{1}{3}\right)+\dfrac{2\i}{\pi}\Li_2\left(e^{-2\pi\i z}\right)
    &\text{if $\Im{z}<0$}.
  \end{cases}
\end{multline*}
\end{lem}
\begin{proof}
First we assume that $\Im{z}\ge0$.
\par
For a real number $r>1$, let $\overline{U}_{r}^{\pm}$ be the vertical segment connecting $\pm re^{\i\theta}$ and $\pm re^{\i\theta}+r\i$, and $\overline{H}_{r}$ be the segment connecting $-re^{\i\theta}+r\i$ and $re^{\i\theta}+r\i$.
Here we assume that $r$ is not an integer multiple of $\pi$ so that $\overline{H}_r$ avoids the poles of $\frac{e^{(2z-1)x}}{x^m\sinh(x)}$ ($m=1,2$) as a function of $x$.
We orient $\overline{U}_{r}^{\pm}$ upward and $\overline{H}_{r}$ from left to right.
Note that the distance between the origin and the line containing $\overline{U}_{r}^{\pm}$ is $r\cos\theta$, and that the distance between the origin and the line containing $\overline{H}_{r}$ is also $r\cos\theta$.
\par
By the residue theorem we have
\begin{equation}\label{eq:L1_residue_positive}
\begin{split}
  &\int_{C_{\theta}^{r}}\frac{e^{(2z-1)x}}{x^m\sinh(x)}\,dx
  -
  \int_{\overline{H}_{r}}\frac{e^{(2z-1)x}}{x^m\sinh(x)}\,dx
  \\
  &+
  \int_{\overline{U}^{+}_{r}}\frac{e^{(2z-1)x}}{x^m\sinh(x)}\,dx
  -
  \int_{\overline{U}^{-}_{r}}\frac{e^{(2z-1)x}}{x^m\sinh(x)}\,dx
  \\
  =&
  2\pi\i
  \sum_{k=1}^{\lfloor{r}\rfloor}
  \Res\left(\frac{e^{(2z-1)x}}{x^m\sinh(x)};x=k\pi\i\right)
\end{split}
\end{equation}
for $m=1,2$, where $\lfloor{r}\rfloor$ is the greatest integer less than or equal to $r$.
Since the order of the pole $x=k\pi\i$ of $\frac{e^{(2t-1)x}}{x^m\sinh(x)}$ for $k=1,2,3,\dots$ is one, we have
\begin{equation*}
  \Res\left(\frac{e^{(2z-1)x}}{x\sinh(x)};x=k\pi\i\right)
  =
  \lim_{t\to k\pi\i}\frac{(x-k\pi\i)e^{(2z-1)x}}{x\sinh(x)}
  =
  \frac{e^{2kz\pi\i}}{k\pi\i}
\end{equation*}
and
\begin{equation*}
  \Res\left(\frac{e^{(2z-1)x}}{x^2\sinh(x)};x=k\pi\i\right)
  =
  \lim_{x\to k\pi\i}\frac{(x-k\i)e^{(2z-1)x}}{x^2\sinh(x)}
  =
  \frac{-e^{2kz\pi\i}}{k^2\pi^2}.
\end{equation*}
\par
Next, we calculate integrals along $\overline{H}_r$ and $\overline{U}_r^{\pm}$.
Note that since the distance between the origin and any point on these segments is greater than or equal to $r\cos\theta$, we have
\begin{align*}
  \left|
    \int_{\overline{H}_r}\frac{e^{(2z-1)x}}{x^m\sinh(x)}\,dx
  \right|
  &\le
  \frac{1}{(r\cos\theta)^m}\left|\int_{\overline{H}_r}\frac{e^{(2z-1)x}}{\sinh(x)}\,dx\right|
  \\
  \intertext{and}
  \left|
    \int_{\overline{U}_r^{\pm}}\frac{e^{(2z-1)x}}{x^m\sinh(x)}\,dx
  \right|
  &\le
  \frac{1}{(r\cos\theta)^m}\left|\int_{\overline{U}_r^{\pm}}\frac{e^{(2z-1)x}}{\sinh(x)}\,dx\right|
\end{align*}
for $m=1,2$.
\par
Putting $x=ye^{\i\theta}+r\i$, we have
\begin{equation*}
\begin{split}
  &\left|\int_{\overline{H}_{r}}\frac{e^{(2z-1)x}}{x^m\sinh(x)}\,dx\right|
  \\
  \le&
  \frac{1}{(r\cos\theta)^m}
  \left|
    \int_{-r}^{r}
    \frac{e^{(2z-1)(ye^{\i\theta}+r\i)}\times e^{\i\theta}}
         {\sinh(ye^{\i\theta}+r\i)}\,dy
  \right|
  \\
  \le&
  \frac{1}{(r\cos\theta)^m}
  \int_{-r}^{r}
  \frac{\left|e^{(2z-1)(ye^{\i\theta}+r\i)}\right|}{\left|\sinh(ye^{\i\theta}+r\i)\right|}\,dy
  \\
  =&
  \frac{e^{-2r\Im{z}}}{(r\cos\theta)^m}
  \int_{-r}^{r}
  \frac{e^{y\Re\left((2z-1)e^{\i\theta}\right)}}
       {\left|\sinh(ye^{\i\theta}+r\i)\right|}\,dy
  \\
  &\text{\small
        ($M
        :=
        \max_{-1\le y\le1}
        \frac{e^{y\Re\left((2z-1)e^{\i\theta}\right)}}{\left|\sinh(ye^{\i\theta}+r\i)\right|}>0$)}
  \\
  \le&
  \frac{e^{-2r\Im{z}}}{(r\cos\theta)^m}
  \\
  &\quad\times
  \left(
    2M
    +
    \int_{-r}^{-1}
    \frac{2e^{2y\Re\left(ze^{\i\theta}\right)-y\cos\theta}}{e^{-y\cos\theta}-e^{y\cos\theta}}\,dy
    +
    \int_{1}^{r}
    \frac{2e^{2y\Re\left(ze^{\i\theta}\right)-y\cos\theta}}{e^{y\cos\theta}-e^{-y\cos\theta}}\,dy
  \right)
  \\
  =&
  \frac{2e^{-2r\Im{z}}}{(r\cos\theta)^m}
  \left(
    M
    +
    \int_{-r}^{-1}
    \frac{e^{2y\Re\left(ze^{\i\theta}\right)}}{1-e^{2y\cos\theta}}\,dy
    +
    \int_{1}^{r}
    \frac{e^{2y\Re\left(ze^{\i\theta}\right)-2y\cos\theta}}{1-e^{-2y\cos\theta}}\,dy
  \right)
  \\
  \le&
  \frac{2e^{-2r\Im{z}}}{(r\cos\theta)^m}
  \left(
    M
    +
    \frac{\int_{-r}^{-1}e^{2y\Re\left(ze^{\i\theta}\right)}\,dy}
         {1-e^{-2\cos\theta}}
    +
    \frac{\int_{1}^{r}e^{2y\Re\left(ze^{\i\theta}\right)-2y\cos\theta}\,dy}
         {1-e^{-2\cos\theta}}
  \right)
  \\
  =&
  \frac{2e^{-2r\Im{z}}}{(r\cos\theta)^m}
  \Biggl(
    M
    +
    \frac{e^{-2\Re\left(ze^{\i\theta}\right)}-e^{-2r\Re\left(ze^{\i\theta}\right)}}
         {2\left(1-e^{-2\cos\theta}\right)\Re\left(ze^{\i\theta}\right)}
  \\
  &\phantom{\frac{4e^{-r\Im{z}}}{(r\cos\theta)^m}\Biggl(}
    \quad+
    \frac{e^{r\left(2\Re\left(ze^{\i\theta}\right)-2\cos\theta\right)}
         -e^{2\Re\left(ze^{\i\theta}\right)-2\cos\theta}}
         {2\left(1-e^{-2\cos\theta}\right)\left(\Re\left(ze^{\i\theta}\right)-\cos\theta\right)}
  \Biggr).
\end{split}
\end{equation*}
This converges to zero as $r\to\infty$ since $0<\Re(ze^{\i\theta})<\cos\theta$ and $\Im{z}\ge0$.
Note that $M$ depends on $r$ but that it is bounded because it is periodic with respect to $r$.
\par
Putting $x=r(e^{\i\theta}+y\i)$, we have
\begin{equation*}
\begin{split}
  \left|
    \int_{\overline{U}_r^{+}}\frac{e^{(2z-1)x}}{x^m\sinh(x)}\,dx
  \right|
  &\le
  \frac{1}{(r\cos\theta)^m}
  \left|
    \int_{0}^{1}\frac{e^{(2z-1)r(e^{\i\theta}+y\i)}\times r\i}{\sinh(r(e^{\i\theta}+y\i))}\,dy
  \right|
  \\
  &\le
  \frac{1}{(r\cos\theta)^m}
  \int_{0}^{1}
  \frac{r\left|e^{(2z-1)r(e^{\i\theta}+y\i)}\right|}{\left|\sinh(r(e^{\i\theta}+y\i))\right|}\,dy
  \\
  &\le
  \frac{2re^{2r\Re\left(ze^{\i\theta}\right)-r\cos\theta}}
       {(r\cos\theta)^m\left(e^{r\cos\theta}-e^{-r\cos\theta}\right)}
  \times
  \int_{0}^{1}e^{-2ry\Im{z}}\,dy
  \\
  &=
  \frac{2e^{2r\Re\left(ze^{\i\theta}\right)-2r\cos\theta}}
       {r^{m-1}\cos^m\theta\left(1-e^{-2r\cos\theta}\right)}
  \times
  \int_{0}^{1}e^{-2ry\Im{z}}\,dy
  \\
  &\to0\quad(r\to\infty)
\end{split}
\end{equation*}
since $\Re(ze^{\i\theta})<\cos\theta$, noting that the last integral becomes either $1$ (if $z$ is real) or $\frac{1-e^{-2r\Im{z}}}{2r\Im{z}}$ (otherwise).
\par
Similarly, putting $x=r(-e^{\i\theta}+y\i)$, we have
\begin{equation*}
\begin{split}
  \left|
    \int_{\overline{U}_r^{-}}\frac{e^{(2z-1)x}}{x^m\sinh(x)}\,dx
  \right|
  &\le
  \frac{1}{(r\cos\theta)^m}
  \left|
    \int_{0}^{1}\frac{e^{-(2z-1)r(e^{\i\theta}-y\i)}\times r\i}{\sinh(r(e^{\i\theta}-y\i))}\,dy
  \right|
  \\
  &\le
  \frac{1}{(r\cos\theta)^m}
  \int_{0}^{1}
  \frac{r\left|e^{-(2z-1)r(e^{\i\theta}-y\i)}\right|}{\left|\sinh(r(e^{\i\theta}-y\i))\right|}\,dy
  \\
  &\le
  \frac{2re^{-2r\Re\left(ze^{\i\theta}\right)+r\cos\theta}}
       {(r\cos\theta)^m\left(e^{r\cos\theta}-e^{-r\cos\theta}\right)}
  \times
  \int_{0}^{1}e^{-2ry\Im{z}}\,dy
  \\
  &=
  \frac{2e^{-2r\Re\left(ze^{\i\theta}\right)}}
       {r^{m-1}\cos^m\theta\left(1-e^{-2r\cos\theta}\right)}
  \times
  \int_{0}^{1}e^{-2ry\Im{z}}\,dy
  \\
  &\to0\quad(r\to\infty)
\end{split}
\end{equation*}
since $0<\Re(ze^{\i\theta})$.
Therefore from \eqref{eq:L1_residue_positive} we have
\begin{equation*}
  \lim_{r\to\infty}
  \int_{C_{\theta}^{r}}\frac{e^{(2z-1)x}}{x\sinh(x)}\,dx
  =
  2
  \lim_{r\to\infty}
  \sum_{k=1}^{\lfloor{r}\rfloor}\frac{e^{2kz\pi\i}}{k},
\end{equation*}
which converges to $-2\log(1-e^{2\pi\i z})$ as $r\to\infty$ when $\left|e^{2z\pi\i}\right|<1$, or $e^{2z\pi\i}\ne1$ ($z\in\R$), that is, when $\Im{z}\ge0$ (Recall that we assume $0<\Re(ze^{\i\theta})<\cos\theta$).
We also have
\begin{equation*}
  \lim_{r\to\infty}
  \int_{C_{\theta}^{r}}\frac{e^{(2z-1)x}}{x^2\sinh(x)}\,dx
  =
  -\frac{2\i}{\pi}
  \lim_{r\to\infty}
  \sum_{k=1}^{\lfloor{r}\rfloor}\frac{e^{2kz\pi\i}}{k^2},
\end{equation*}
which converges to $-\frac{2\i}{\pi}\Li_2(e^{2\pi\i z})$ as $r\to\infty$ when $\Im{z}\ge0$ from Lemma~\ref{lem:Li_2_convergence} below.
\par
This completes the case where $\Im{z}\ge0$.
\par
Next we assume that $\Im{z}<0$.
\par
For $r>1$, let $\underline{U}_{r}^{\pm}$ be the vertical segment connecting $\pm re^{\i\theta}$ and $\pm re^{\i\theta}-r\i$, and $\underline{H}_{r}$ be the segment connecting $-re^{\i\theta}-r\i$ and $re^{\i\theta}-r\i$.
We orient $\underline{U}_{r}^{\pm}$ upward and $\underline{H}_{r}$ from left to right.
Note that the distance between the origin and the line containing $\underline{U}_{r}^{\pm}$ is $r\cos\theta$, and that the distance between the origin and the line containing $\underline{H}_{r}$ is also $r\cos\theta$.
\par
For $m=1,2$, we have
\begin{equation}\label{eq:L1_residue_negative}
\begin{split}
  &-\int_{C_{\theta}^{r}}\frac{e^{(2z-1)x}}{x^m\sinh(x)}\,dx
  +
  \int_{\underline{H}_{r}}\frac{e^{(2z-1)x}}{x^m\sinh(x)}\,dx
  \\
  &+
  \int_{\underline{U}_{r}^{+}}\frac{e^{(2z-1)x}}{x^m\sinh(x)}\,dx
  -
  \int_{\underline{U}_{r}^{-}}\frac{e^{(2z-1)x}}{x^m\sinh(x)}\,dx
  \\
  =&
  2\pi\i
  \sum_{k=0}^{\lfloor{r}\rfloor}
  \Res\left(\frac{e^{(2z-1)x}}{x^m\sinh(x)};x=-k\pi\i\right).
\end{split}
\end{equation}
Since the order of the pole $x=-k\pi\i$ of $\frac{e^{(2z-1)}}{x^m\sinh(x)}$ ($m=1,2$) for $k=1,2,3,\dots$ is one, we have
\begin{align*}
  \Res\left(\frac{e^{(2z-1)x}}{x\sinh(x)};x=-k\pi\i\right)
  =
  \lim_{x\to-k\pi\i}\frac{(x+k\pi\i)e^{(2z-1)x}}{x\sinh(x)}
  =
  \frac{-e^{-2k\pi\i}}{k\pi\i}
  \\
  \intertext{and}
  \Res\left(\frac{e^{(2z-1)x}}{x^2\sinh(x)};x=-k\pi\i\right)
  =
  \lim_{x\to-k\i}\frac{(x+k\i)e^{(2z-1)x}}{x^2\sinh(x)}
  =
  \frac{-e^{-2k\pi\i}}{k^2\pi^2}.
\end{align*}
Since $e^{(2z-1)x}=1+(2z-1)x+\frac{(2z-1)^2x^2}{2}+\cdots$ and $\frac{1}{\sinh(x)}=\frac{1}{x}-\frac{x}{6}+\cdots$, we have
\begin{align*}
  \Res\left(\frac{e^{(2z-1)x}}{x\sinh(x)};x=0\right)
  &=
  2z-1
  \\
  \intertext{and}
  \Res\left(\frac{e^{(2z-1)x}}{x^2\sinh(x)};x=0\right)
  &=
  2z^2-2z+\frac{1}{3}.
\end{align*}
\par
We can prove the integrals along $\underline{H}_r$ and $\underline{U}_{r}^{\pm}$ converge to zero as $r\to\infty$ in similar ways to the cases of $\overline{H}_r$ and $\overline{U}_{r}^{\pm}$.
\par
Putting $x=ye^{\i\theta}-r\i$ and assuming $r>1$, we have
\begin{equation*}
\begin{split}
  &\left|\int_{\underline{H}_{r}}\frac{e^{(2z-1)x}}{x^m\sinh(x)}\,dx\right|
  \\
  \le&
  \frac{1}{(r\cos\theta)^m}
  \left|
    \int_{-r}^{r}
    \frac{e^{(2z-1)(ye^{\i\theta}-r\i)}\times e^{\i\theta}}
         {\sinh(ye^{\i\theta}-r\i)}\,dy
  \right|
  \\
  \le&
  \frac{e^{2r\Im{z}}}{(r\cos\theta)^m}
  \int_{-r}^{r}
  \frac{e^{y\Re\left((2z-1)e^{\i\theta}\right)}}
       {\left|\sinh(ye^{\i\theta}-r\i)\right|}\,dy
  \\
  &\text{\small
    ($M
    :=
    \max_{-1\le y\le1}
    \frac{e^{y\Re\left((2z-1)e^{\i\theta}\right)}}{\left|\sinh(ye^{\i\theta}-r\i)\right|}>0$)}
  \\
  \le&
  \frac{e^{2r\Im{z}}}{(r\cos\theta)^m}
  \\
  &\quad\times
  \left(
    2M
    +
    \int_{-r}^{-1}
    \frac{2e^{2y\Re\left(ze^{\i\theta}\right)-y\cos\theta}}{e^{-y\cos\theta}-e^{y\cos\theta}}\,dy
    +
    \int_{1}^{r}
    \frac{2e^{2y\Re\left(ze^{\i\theta}\right)-y\cos\theta}}{e^{y\cos\theta}-e^{-y\cos\theta}}\,dy
  \right),
\end{split}
\end{equation*}
which converges to zero as $r\to\infty$.
\par
Putting $x=r(e^{\i\theta}-y\i)$, we have
\begin{equation*}
\begin{split}
  \left|\int_{\underline{U}_r^{+}}\frac{e^{(2z-1)x}}{x^m\sinh(x)}\,dx\right|
  &\le
  \frac{1}{(r\cos\theta)^m}
  \left|
    \int_{0}^{1}
    \frac{e^{(2z-1)r(e^{\i\theta}-y\i)}\times(-r\i)}
         {\sinh(r(e^{\i\theta}-y\i))}\,dy
  \right|
  \\
  &\le
  \frac{2re^{2r\Re\left(ze^{\i\theta}\right)-r\cos\theta}}
       {(r\cos\theta)^m(e^{r\cos\theta}-e^{-r\cos\theta})}
  \int_{0}^{1}e^{2ry\Im{z}}\,dy
  \\
  &=
  \frac{2e^{2r\Re\left(ze^{\i\theta}\right)-2r\cos\theta}}
       {r^{m-1}\cos^m\theta(1-e^{-2r\cos\theta})}
  \times
  \int_{0}^{1}e^{2ry\Im{z}}\,dy,
\end{split}
\end{equation*}
since $0<\Re(ze^{\i\theta})<\cos\theta$, noting that the last integral becomes either $1$ (if $z$ is real) or $\frac{e^{2r\Im{z}}-1}{2r\Im{z}}$ (otherwise).
\par
Similarly, putting $x=-r(e^{\i\theta}+y\i)$, we have
\begin{equation*}
\begin{split}
  \left|
    \int_{\underline{U}_{r}^{-}}\frac{e^{(2z-1)x}}{x^m\sinh(x)}\,dx
  \right|
  &\le
  \frac{1}{(r\cos\theta)^m}
  \left|
    \int_{0}^{1}
    \frac{e^{-r(2z-1)(e^{\i\theta}+y\i)}\times(-r\i)}
         {\sinh(-r(e^{\i\theta}+y\i))}\,dy
  \right|
  \\
  &\le
  \frac{2re^{-2r\Re\left(ze^{\i\theta}\right)+r\cos\theta}}
       {(r\cos\theta)^m(e^{r\cos\theta}-e^{-r\cos\theta})}
  \int_{0}^{1}e^{2ry\Im{z}}\,dy
  \\
  &=
  \frac{2e^{-2r\Re\left(ze^{\i\theta}\right)}}
       {r^{m-1}\cos^m\theta(1-e^{-2r\cos\theta})}
  \times
  \int_{0}^{1}e^{2ry\Im{z}}\,dy
  \\
  &\to0\quad(r\to\infty)
\end{split}
\end{equation*}
since $0<\Re(ze^{\i\theta})$.
So from \eqref{eq:L1_residue_negative} we have
\begin{equation*}
\begin{split}
  \lim_{r\to\infty}
  \int_{C_{\theta}^{r}}\frac{e^{(2z-1)x}}{x\sinh(x)}\,dx
  &=
  -2\pi\i
  \lim_{r\to\infty}
  \sum_{k=0}^{\lfloor{r}\rfloor}\Res\left(\frac{e^{(2z-1)x}}{x\sinh(x)};x=-k\pi\i\right).
  \\
  &=
  -2\pi\i(2z-1)
  +
  2\lim_{r\to\infty}\sum_{k=1}^{\lfloor{r}\rfloor}\frac{e^{-2kz\pi\i}}{k}.
\end{split}
\end{equation*}
Since this series converges to $-\log\left(1-e^{-2z\pi\i}\right)$ as $r\to\infty$ if $\Im{z}<0$ from Lemma~\ref{lem:Li_2_convergence}, we finally have
\begin{equation*}
  \int_{C_{\theta}}
  \frac{e^{(2z-1)x}}{x\sinh(x)}\,dt
  =
  -2\pi\i(2z-1)-2\log\left(1-e^{-2z\pi\i}\right),
\end{equation*}
completing the proof when $\Im{z}<0$.
\par
Similarly, we have
\begin{equation*}
\begin{split}
  \lim_{r\to\infty}
  \int_{C_{\theta}^{r}}\frac{e^{(2z-1)x}}{x^2\sinh(x)}\,dx
  &=
  -2\pi\i
  \lim_{r\to\infty}
  \sum_{k=0}^{\lfloor{r}\rfloor}\Res\left(\frac{e^{(2z-1)x}}{x^2\sinh(x)};x=-k\pi\i\right).
  \\
  &=
  -2\pi\i\left(2z^2-2z+\frac{1}{3}\right)
  +
  \frac{2\i}{\pi}\lim_{r\to\infty}\sum_{k=1}^{\lfloor{r}\rfloor}\frac{e^{-2kz\pi\i}}{k^2}.
\end{split}
\end{equation*}
The series converges to $\Li_2\left(e^{-2z\pi\i}\right)$ and so
\begin{equation*}
  \int_{C_{\theta}}
  \frac{e^{(2z-1)x}}{x^2\sinh(x)}\,dt
  =
  -2\pi\i\left(2z^2-2z+\frac{1}{3}\right)+\frac{2\i}{\pi}\Li_2\left(e^{-2z\pi\i}\right),
\end{equation*}
completing the proof when $\Im{z}<0$.
\end{proof}
We give a proof for the following well-known lemma.
\begin{lem}\label{lem:Li_2_convergence}
For a complex number $w$ with $|w|\le1$, the series $\sum_{k=1}^{\infty}\frac{w^k}{k^2}$ converges to $\Li_2(w)$.
Here we use $\Li_2(w):=-\int_{0}^{w}\frac{\log(1-t)}{t}\,dt$ as the definition of the dilogarithm.
\end{lem}
\begin{proof}
Put $a_k:=\frac{1}{k^2}$.
\par
Then we have
\begin{equation*}
  \left|\frac{a_{k+1}}{a_{k}}\right|
  =
  \left(\frac{k}{k+1}\right)^2
  \to
  1
\end{equation*}
as $k\to\infty$.
Therefore from d'Alembert's ratio test, the radius of convergence of the power series $\sum_{k=1}^{\infty}a_kw^k$ is $1$.
So we can differentiate it term by term if $|w|<1$, and we obtain
\begin{equation*}
  \frac{d}{d\,w}
  \left(\sum_{k=1}^{\infty}\frac{w^k}{k^2}\right)
  =
  \sum_{k=1}^{\infty}\frac{w^{k-1}}{k}
  =
  -\frac{\log(1-w)}{w}.
\end{equation*}
Therefore we conclude that $\sum_{k=1}^{\infty}\frac{w^k}{k^2}=-\int_{0}^{w}\frac{\log(1-t)}{t}\,dt$.
\par
If $w=1$, the series $\sum_{k=1}^{\infty}a_k$ converges by the integral test.
\par
Finally, we assume that $|w|=1$ with $w\ne1$, and apply Abel's test.
Since the sequence $\{a_k\}$ is positive and monotonically decreasing with $\lim_{k\to\infty}a_k=0$, the power series $\sum_{k=1}^{\infty}a_kw^k$ converges if $|w|=1$ ($w\ne1$).
We can also apply Abel's theorem to conclude that $\sum_{k=1}^{\infty}a_kz^k$ converges to $\sum_{k=1}^{\infty}a_kw^k$ provided that $z$ approaches $w$ along the radius.
Note that this includes the case $w=1$.
Therefore we also conclude that $\sum_{k=1}^{\infty}\frac{w^k}{k^2}=-\int_{0}^{w}\frac{\log(1-t)}{t}\,dt$ for the case $|w|=1$ by choosing the integral path as the radius connecting $0$ and $w$.
\end{proof}
\par
We prove that $\frac{1}{N}T_{N}(z)$ uniformly converges to $\frac{1}{\xi}\L_{2}(z)$.
\begin{proof}[Proof of Proposition~\ref{prop:T_N_L_2}]
We have
\begin{equation*}
\begin{split}
  \left|
    T_N(z)
    -
    \frac{N}{\xi}\L_2(z)
  \right|
  &=
  \frac{1}{4}
  \left|
    \int_{C_{\theta}}
    \left(
      \frac{e^{(2z-1)x}}{x\sinh(x)\sinh(\gamma x)}
      -
      \frac{e^{(2z-1)x}}{\gamma x^2\sinh(x)}
    \right)
    \,dx
  \right|
  \\
  &=
  \frac{1}{4}
  \left|
    \int_{C_{\theta}}
    \frac{e^{(2z-1)x}}{\gamma x^2\sinh(x)}
    \left(
      \frac{\gamma x}{\sinh(\gamma x)}-1
    \right)
    \,dx
  \right|
  \\
  &\le
  \frac{\pi N}{2|\xi|}
  \int_{C_{\theta}}
  \left|
    \frac{e^{(2z-1)x}}{x^2\sinh(x)}
  \right|
  \left|
    \frac{\gamma x}{\sinh(\gamma x)}-1
  \right|
  \,dx.
\end{split}
\end{equation*}
Since $\frac{\gamma x}{\sinh(\gamma x)}=1-\frac{(\gamma x)^2}{6}+O((\gamma x)^4)$, $\left|\frac{\gamma x}{\sinh(\gamma x)}-1\right|<\frac{C|x|^2}{N^2}$ for a positive constant $C$ since $\gamma=\frac{\xi}{2N\pi\i}$.
So we have
\begin{equation*}
  \left|
    T_N(z)
    -
    \frac{N}{\xi}\L_2(z)
  \right|
  \le
  \frac{C'}{N}
  \int_{C_{\theta}}
  \left|\frac{e^{(2z-1)x}}{\sinh(x)}\right|\,dx,
\end{equation*}
where $C':=\frac{C\pi}{2|\xi|}$.
Put
\begin{align*}
  I_{+}
  &:=
  \lim_{r\to\infty}
  \int_{e^{\i\theta}}^{re^{\i\theta}}
  \left|\frac{e^{(2z-1)x}}{\sinh(x)}\right|\,dx,
  \\
  I_{-}
  &:=
  \lim_{r\to\infty}
  \int_{-re^{\i\theta}}^{-e^{\i\theta}}
  \left|\frac{e^{(2z-1)x}}{\sinh(x)}\right|\,dx,
  \\
  I_{0}
  &:=
  \int_{|x|=1,\theta\le\arg{x}\le\pi+\theta}
  \left|\frac{e^{(2z-1)x}}{\sinh(x)}\right|\,dx.
\end{align*}
Putting $x:=e^{\i\theta}y$, we have
\begin{equation*}
\begin{split}
  I_{+}
  &=
  \int_{1}^{\infty}
  \frac{\left|e^{(2z-1)ye^{\theta\i}}\right|}{|\sinh(ye^{\i\theta})|}
  \,dy
  \\
  &\le
  \int_{1}^{\infty}
  \frac{2e^{2y\Re(ze^{\theta\i})-y\cos\theta}}{e^{y\cos\theta}-e^{-y\cos\theta}}
  \,dy
  \\
  &=
  \int_{1}^{\infty}
  \frac{2e^{2y\left(\Re(ze^{\theta\i})-\cos\theta\right)}}{1-e^{-2y\cos\theta}}
  \,dy
  \\
  &\le
  \frac{2}{1-e^{-2\cos\theta}}
  \int_{1}^{\infty}
  e^{-2y\nu}\,dy
  \\
  &=
  \frac{e^{-2\nu}}{(1-e^{-2\cos\theta})\nu}.
\end{split}
\end{equation*}
Here we use the assumption $\Re(ze^{\theta\i})\le\cos\theta-\nu$.
\par
Similarly, we have
\begin{equation*}
\begin{split}
  I_{-}
  &=
  \int_{-\infty}^{-1}
  \frac{\left|e^{(2z-1)ye^{\theta\i}}\right|}{|\sinh(ye^{\i\theta})|}
  \,dy
  \\
  &\le
  \int_{-\infty}^{-1}
  \frac{2e^{2y\Re(ze^{\theta\i})-y\cos\theta}}{e^{-y\cos\theta}-e^{y\cos\theta}}
  \,dy
  \\
  &=
  \int_{-\infty}^{-1}
  \frac{2e^{2y\Re(ze^{\theta\i})}}{1-e^{2y\cos\theta}}
  \,dy
  \\
  &\le
  \frac{2}{1-e^{-2\cos\theta}}
  \int_{-\infty}^{-1}
  e^{2y\nu}\,dy
  \\
  &=
  \frac{e^{2\nu}}{(1-e^{-2\cos\theta})\nu}.
\end{split}
\end{equation*}
Here we use the assumption $\Re(ze^{\theta\i})\ge\nu$.
\par
Finally, putting $x=e^{\tau\i}$ ($\theta\le\tau\le\theta+\pi$) and $L:=\min_{\theta\le\tau\le\theta+\pi}|\sinh(e^{\i\tau})|>0$, we have
\begin{equation*}
\begin{split}
  I_{0}
  &=
  \int_{\theta}^{\theta+\pi}
  \frac{\left|e^{(2z-1)e^{\tau\i}}\right|}{|\sinh(e^{\i\tau})|}
  |\i e^{\tau\i}|\,d\tau
  \\
  &\le
  \frac{1}{L}
  \int_{\theta}^{\theta+\pi}
  e^{2\Re(ze^{\tau\i})-\cos\tau}
  \,d\tau
  \\
  &=
  \frac{1}{L}
  \int_{\theta}^{\theta+\pi}
  e^{(2\Re{z}-1)\cos\tau-2\sin\tau\Im{z}}
  \,d\tau,
\end{split}
\end{equation*}
which is bounded from the above because both $\Re{z}$ and $\Im{z}$ are bounded.
\par
Therefore we conclude that $\left|T_N(z)-\frac{N}{\xi}\L_2(z)\right|\le\frac{C''}{N}$ for some constant $C''$ that does not depend on $z$.
\end{proof}
\begin{proof}[Proof of Lemma~\ref{lem:T_special}]
Recall that $\gamma=\frac{\xi}{2N\pi\i}$.
\par
By the definition of $T_N(z)$, we have
\begin{equation*}
\begin{split}
  &
  T_{N}\left(\frac{\xi}{2\pi\i}(1+\frac{1}{2N})\right)
  -
  T_{N}\left(\frac{\xi}{2\pi\i}(1-\frac{1}{2N})\right)
  \\
  =&
  \frac{1}{4}\int_{C_{\theta}}
  \frac{e^{\left(\frac{\xi}{\pi\i}(1+\frac{1}{2N})-1\right)x}
       -e^{\left(\frac{\xi}{\pi\i}(1-\frac{1}{2N})-1\right)x}}
       {x\sinh(x)\sinh(\gamma x)}
  \,dx
  \\
  =&
  \frac{1}{2}
  \int_{C_{\theta}}\frac{e^{\left(\frac{\xi}{\pi\i}-1\right)x}}{x\sinh(x)}
  \,dx,
\end{split}
\end{equation*}
which equals $\pi\i-\xi-\log(1-e^{-\xi})$ from Lemma~\ref{lem:L1_2}.
\begin{equation*}
  \frac{\exp\left(T_{N}\left(\frac{\xi}{2\pi\i}(1+\frac{1}{2N})\right)\right)}
       {\exp\left(T_{N}\left(\frac{\xi}{2\pi\i}(1-\frac{1}{2N})\right)\right)}
  =
  \frac{1}{1-e^{\xi}}.
\end{equation*}
\end{proof}
\begin{rem}
The proof of Lemma~2.3 in \cite{Murakami:JTOP2013} is wrong, which was informed by Ka Ho Wong.
\end{rem}
\bibliography{mrabbrev,hitoshi}
\bibliographystyle{amsplain}
\end{document}